\documentclass[11 pt]{amsart}
\usepackage{amscd}
\usepackage{amsfonts}
\usepackage{amsmath}
\usepackage{amssymb}
\usepackage{amsthm}
\usepackage{latexsym}
\usepackage{eucal}
\usepackage{hyperref}
\usepackage{url}
\usepackage{mathrsfs}
\usepackage{newcent}
\usepackage[all]{xy}
\usepackage{color}

\newcommand{\Spec}{\operatorname{Spec}}
\newcommand{\Spf}{\operatorname{Spf}}
\newcommand{\et}{\operatorname{\acute{e}t}}
\newcommand{\cris}{\operatorname{cris}}
\newcommand{\DR}{\operatorname{dR}}
\newcommand{\doubp}{\operatorname{dp}}

\newcommand{\bbar}{\overline{b}}
\newcommand{\sbar}{\overline{s}}

\newcommand{\Pic}{\mathrm{Pic}}
\newcommand{\Cl}{\mathrm{Cl}}
\newcommand{\QQ}{\mathbb{Q}}
\newcommand{\CC}{\mathbb{C}}
\newcommand{\FF}{\mathbb{F}}
\newcommand{\RR}{\mathbb{R}}

\newcommand{\LL}{\mathbb{L}}

\newcommand{\PP}{\mathbb{P}}
\newcommand{\DD}{\mathbb{D}}
\newcommand{\ZZ}{\mathbb{Z}}
\newcommand{\KK}{\mathbb{K}}
\newcommand{\GG}{\mathbb{G}}

\newcommand{\calO}{\mathcal{O}}
\newcommand{\calA}{\mathcal{A}}
\newcommand{\calB}{\mathcal{B}}
\newcommand{\calC}{\mathcal{C}}

\newcommand{\calE}{\mathcal{E}}
\newcommand{\calF}{\mathcal{F}}

\newcommand{\calL}{\mathcal{L}}

\newcommand{\calV}{\mathcal{V}}
\newcommand{\calW}{\mathcal{W}}
\newcommand{\calX}{\mathcal{X}}
\newcommand{\calY}{\mathcal{Y}}
\newcommand{\End}{\mathrm{End}}
\newcommand{\Hom}{\mathrm{Hom}}

\newcommand{\Rbar}{\overline{R}}

\newcommand{\DDD}{\mathsf{D}}

\newcommand{\M}{\mathsf{M}}
\newcommand{\F}{\mathsf{F}}
\newcommand{\A}{\mathsf{A}}

\newcommand{\Del}{\nabla}
\newcommand{\Pet}{P^{2}_{\et}f_{*}}
\newcommand{\Ret}{R_{\et}}
\newcommand{\calAhat}{\widehat{\calA}}
\newcommand{\Abar}{\overline{\widehat{A}}}
\newcommand{\ZZhat}{\widehat{\ZZ}}
\newcommand{\Ahat}{\widehat{A}}

\newcommand{\Shat}{\widehat{S}}
\newcommand{\Bhat}{\widehat{B}}

\newcommand{\Xhat}{\widehat{X}}
\newcommand{\kk}{\mathsf{k}}
\newcommand{\pr}{\mathsf{pr}}

\newcommand{\nodivide}{\nmid}
\newcommand{\NE}{\mathrm{NE}}
\newcommand{\Fil}{\mathrm{Fil}}

\newcommand{\FL}{\mathrm{FL}}
\newcommand{\uT}{\underline{\rm T}}

\newcommand{\bw}[1]{{\textstyle \bigwedge}^{#1}}
\newcommand{\st}{\mathrm{st}}
\DeclareMathOperator{\Sym}{Sym}
\DeclareMathOperator{\rk}{rk}
\DeclareMathOperator{\charr}{char}
\DeclareMathOperator{\disc}{disc}

\let\mf\mathfrak
\let\mc\mathcal

\let\wh\widehat
\newcommand{\simrightarrow}{\overset{\sim}{\rightarrow}}

\theoremstyle{plain}
\newtheorem{lem}{Lemma}[section]
\newtheorem{thm}[lem]{Theorem}
\newtheorem*{conj}{Conjecture}

\newtheorem*{defin}{Definition}
\newtheorem*{mainthm}{Main Theorem}
\newtheorem{prop}[lem]{Proposition}

\newtheorem{cor}[lem]{Corollary}
\newtheorem*{corollary}{Corollary}

\theoremstyle{definition}
\newtheorem{defn}[lem]{Definition}

\newtheorem{example}[lem]{Example}

\newtheorem{remark}[lem]{Remark}

\title[Supersingul\"{a}re K3-fl\"{a}chen f\"{u}r grosse primzahlen]{Supersingular K3 surfaces for large primes}
\author[D. Maulik]{Davesh Maulik with an appendix by Andrew Snowden}
\email{dmaulik@math.columbia.edu}
\email{asnowden@math.mit.edu}

\begin{document}

\begin{abstract}
Given a K3 surface $X$ over a field of characteristic $p$, Artin conjectured that if $X$ is supersingular (meaning infinite height) then its Picard rank is 22.  Along with work of Nygaard-Ogus, this conjecture implies the Tate conjecture for K3 surfaces over finite fields with $p \geq 5$.
We prove Artin's conjecture under the additional assumption that $X$ has a polarization of degree $2d$ with $p > 2d+4$. 
Assuming semistable reduction for surfaces in characteristic $p$, we can improve the main result to K3 surfaces which admit a polarization of degree prime-to-$p$ when $p \geq 5$.  

The argument uses Borcherds' construction of automorphic forms on $O(2,n)$ to construct ample divisors on the moduli space.  We also establish finite-characteristic versions of the positivity of the Hodge bundle and the Kulikov-Pinkham-Persson classification of K3 degenerations.  In the appendix by A. Snowden, a compatibility statement is proven between Clifford constructions and integral $p$-adic comparison functors.
\end{abstract}
\date{March 22, 2012}
\maketitle
\tableofcontents
\section{Introduction}

Let $k$ be an algebraically closed field.
Recall that a K3 surface $X$ over $k$ is a smooth projective surface such that the canonical bundle $K_X$ is trivial and $H^1(X,\calO_X) = 0$.
It follows from the injectivity of the Chern class map
$$c_1^{\et}: \Pic(X) \rightarrow H^{2}_{\et}(X,\ZZ_l(1))$$
in \'etale cohomology for $l \ne \mathrm{char} k$
that 
$$\rk \Pic(X) \leq 22 = \rk H^{2}_{\et}(X,\ZZ_l(1)).$$

When $k = \CC$, Hodge theory implies that this bound can be improved to
$$\rk \Pic(X) \leq 20 = h^{1,1}(X).$$  However, when $\charr k = p > 0$, this stronger bound no longer holds.  For example, when $p \equiv 3\bmod 4$, Tate \cite{tate} showed that the Fermat
quartic
$$\{x^4+y^4 + z^4+w^4 = 0 \} \subset \PP_{k}^{3}$$
has Picard rank $22$.

In his 1974 paper \cite{artin-ss}, Artin conjectured the following cohomological criterion for when $X$ has Picard rank 22.
We first recall from \cite{artin-mazur} that the formal Brauer group $\widehat{\mathrm{Br}}(X)$ of $X$ is a one-dimensional formal group scheme representing the functor:
$$T \mapsto \left[\mathrm{Ker}(\mathrm{Br}(X \times T) \rightarrow \mathrm{Br}(X))\right],$$
on finitely generated local Artin $k$-algebras,
where $$\mathrm{Br}(X\times T) = H^2_{\et}(X\times T, \GG_m)$$ is the Brauer group of $X\times T$.

\begin{defin}
A K3 surface $X$ over $k$ is supersingular if the formal Brauer group has infinite height, i.e. $$\widehat{\mathrm{Br}}(X) = \widehat{\GG_a}.$$
\end{defin}
An equivalent formulation (due to \cite{artin-mazur}) is that a K3 surface $X$ is supersingular if the slopes of the Frobenius action on the crystalline cohomology group $$H^{2}_{\cris}(X,W(k))$$ are identically equal to $1$.

Artin's conjecture is that this condition is equivalent to having rank $22$:
\begin{conj}
A K3 surface over $k$ has Picard rank $22$ if and only if it is supersingular.
\end{conj}

The only-if direction of the above conjecture follows from \cite{artin-mazur}, where they bound the Picard rank in terms of the height of $\widehat{\mathrm{Br}}(X)$.  The interesting direction is showing that supersingularity implies every class is algebraic.
In \cite{artin-ss}, Artin proves this conjecture when $X$ admits the structure of an elliptic fibration, modifying 
the argument of \cite{artin-sd} for the Tate conjecture in this setting.  Later, in \cite{rsz}, Rudakov--Zink--Shafarevich prove Artin's conjecture for K3 surfaces admitting a polarization of degree $2$.

In this paper, we show the following:
\begin{mainthm}
If $X$ is a supersingular K3 surface admitting a polarization of degree $2d$ with $p > 2d+ 4$, then Artin's conjecture holds for $X$.
If we assume semistable reduction for smooth projective surfaces over discrete valuation fields with residue field $k$, then Artin's conjecture holds if $p \geq 5$ and there exists a polarization on $X$ of degree prime-to-$p$.
\end{mainthm}
By semistable reduction over a discrete valuation field, we mean the existence of a semistable model over the valuation ring after a finite base change.

When $k = \overline{\FF}_p$, 
recall that a K3 surface $X$ over $k$ satisfies the Tate conjecture if, for every model $X'$ defined over $\FF_{p^{r}}$,
the map
$$c_{1}^{\et}: \Pic(X') \otimes \QQ_l \rightarrow H^2_{\et}(X, \QQ_l(1))^{\mathrm{Gal}(k/\FF_{p^{r}})}$$
is an isomorphism.
When $p \geq 5$, the Tate conjecture is known for K3 surfaces of finite height by work of Nygaard-Ogus \cite{nygaard-ogus}, so we have the following corollary:

\begin{corollary}
The Tate conjecture holds for K3 surfaces admitting a polarization of degree $2d$ such that $p > 2d+ 4$.  Assuming semistable reduction and $p \geq 5$, the Tate conjecture holds for K3 surfaces admitting a polarization of degree prime-to-$p$.
\end{corollary}

The basic idea of the proof is simple.
In his paper (\cite{artin-ss}, Thm.~$1.1$), Artin shows the following striking fact{\footnote{See \cite{johan-jumping} for an argument using crystalline instead of flat cohomology.}}:
in a connected family of supersingular K3 surfaces, although the Picard group can jump under specialization, the rank of the Picard group must remain constant.
Notice that, over $\CC$, such a statement is never true for non-isotrivial families of K3 surfaces, by \cite{oguiso}.

Using this fact and Artin's conjecture for elliptic K3 surfaces, it suffices to show that every connected component of the supersingular locus in the moduli space intersects the elliptic locus nontrivially.  This would force the Picard rank to be $22$ at one point and thus every point.
To show this, we proceed as follows:
\begin{enumerate}
\item  Using work of Borcherds, an elementary bound on coefficients of cusp forms shows the existence of an automorphic form on the moduli space of K3 surfaces (over $\CC$) whose zeroes and poles lie on the locus of elliptic K3 surfaces.  This gives an ample divisor supported on the elliptic locus which, in particular, must intersect any non-isotrivial family of K3 surfaces over a complete base.  

\item  To apply this in characteristic $p$, we use ideas from the minimal model program to show that any connected component of the supersingular locus contains complete curves.  This is where the assumption of either $p > 2d+4$ or semistable reduction is required.  

\item Finally, we show that the extension of our ample divisor to characteristic $p$ remains ample; this requires proving positivity of the Hodge bundle in characteristic $p$ when $p \nodivide 2d$ and $p \geq 5$.  This step is the most technically involved part of the paper as it involves comparison theorems in $p$-adic Hodge theory and the Kuga--Satake construction for K3 surfaces.  If we could lift the curves in step $2$ to characteristic zero, then this step would be unnecessary; however, in general this is impossible (e.g. the supersingular locus contains nonliftable rational curves).
\end{enumerate}

In the second step, we extend the classification of one-parameter semistable degenerations of K3 surfaces of Kulikov-Pinkham-Persson (\cite{kulikov,pinkham-persson}) to characteristic $p$ when $p > 2d+4$.   These results may be of independent interest, e.g., in extending Olsson's work on modular compactifications \cite{olsson-k3} to characteristic $p$.  

There are also geometric corollaries of the results in the third step.  For instance, it implies that the moduli space of polarized K3 surfaces in characteristic $p$ is quasiprojective, when $p \nodivide 2d$ and $ p \geq 5$, which seems to be new to the literature.  Using \cite{vandergeer-katsura}, we also show that the height of the formal Brauer group must jump in a proper nonisotrivial family of finite-height K3 surfaces.

As we were completing this manuscript, we learned that Keerthi Madapusi Pera has announced a proof of the Tate conjecture for K3 surfaces with prime-to-$p$ polarization, without any assumption on semistable reduction.  As his techniques seem largely different from ours, building on Kisin's work on integral models of Shimura varieties, we hope the geometric approach in this paper is still of interest.

\removelastskip\vskip.5\baselineskip\par\noindent{\bf Outline of paper.}
In Section $2$, we set up some notation and basic facts regarding K3 surfaces and Clifford algebras.  We also explain the construction of the moduli space of quasipolarized K3 surfaces in mixed characteristic.  
In Section $3$, we prove the statement regarding ample divisors on the elliptic locus in characteristic zero.  In fact, we give a more general version showing the existence of ample divisors supported on any infinite collection of Noether-Lefschetz loci.
In Section $4$, we prove the KPP classification when $p > 2d+4$ and apply it to show that supersingular K3 surfaces do not degenerate over discrete valuation rings.  In Section $5$, we prove the ampleness of the Hodge bundle on the moduli space of K3 surfaces.  This requires proving that a mixed-characteristic version of the Kuga--Satake morphism is quasifinite, which we do in Section $6$, using comparison theorems in $p$-adic Hodge theory.  In Section $7$, we explain how these steps give the proof of the main theorem.    Finally, in the appendix, Snowden shows a compatibility statement between the Clifford algebra functor and the $p$-adic comparison functors used in Section $6$.

% don't want this appearing in toc so can't use subsection*
\removelastskip\vskip.5\baselineskip\par\noindent{\bf Acknowledgments.}
We would like to thank Bhargav Bhatt, Alessio Corti, Johan de Jong, Igor Dolgachev, H\'{e}l\`{e}ne Esnault, Robert Friedman, Daniel Huybrechts, J\'anos Koll\'ar, Max Lieblich, Ben Moonen, Bjorn Poonen, Jordan Rizov, Matthias Schuett, Olivier Wittenberg, and Chenyang Xu for many helpful comments and discussions.  Thanks also to Snowden for providing the appendix.  

We are especially grateful to Christian Liedtke for many conversations early on in this project.  Among other things, it was his idea that Borcherds' work in characteristic zero should be relevant to Artin's conjecture, and our paper is based on this key insight.
The author is partially supported by a Clay Research Fellowship.

\section{Preliminaries}

In this section, we review some basic constructions regarding K3 surfaces that we will need later.

\subsection{Families of K3 surfaces}
Given a K3 surface $X$ over an algebraically closed field $k$, recall that a line bundle $L$ on $X$ is a polarization if it is ample. 
Similarly, $L$ is a quasipolarization if it is nef and big.  Equivalently, there exists $N$ such that $L^{\otimes N}$ is globally generated and the associated map $X \rightarrow |L^{\otimes N}|$ is
birational onto its image.  The degree of a quasipolarization is
$\deg c_1(L)^{2}\in \ZZ$; it is always a positive even integer.  A line bundle $L$ is primitive if it is not of the form $(L')^{\otimes m}$ for $m >1$.

Given a scheme $S$, let
$f: X \rightarrow S$
be a smooth, proper map from an algebraic space $X$.  We say it is a family of K3 surfaces over $S$ if for each geometric point
$s = \Spec K \rightarrow S$ with $K$ an algebraically closed field, the fiber
$X_s$ is a K3 surface over $K$.

Given an algebraic space $Y$, let $\Pic(Y)$ denote the Picard group of line bundles on $Y$.  Given a family of K3 surfaces $f: X \rightarrow S$, let $\underline{\Pic}_{X/S} \rightarrow S$ denote the Picard functor associated to $S$; by definition, it is the sheafification in the \'etale topology of the functor 
sending $T \rightarrow S$ to $\Pic(X\times_{S}T)/\Pic(T)$.  The sheaf $\underline{\Pic}_{X/S}$ is represented by an algebraic space over $S$.

Given an element $\xi \in \underline{\Pic}_{X/S}(S)$, it defines a polarization
on $f: X \rightarrow S$
if, for each geometric point $s \rightarrow S$, the pullback
$\xi_s \in \underline{\Pic}_{X/S}(s) = \Pic(X_s)$
is a polarization in the sense above.
Similarly, we say it is a quasipolarization (resp. primitive) if its pullback to each geometric fiber is a quasipolarization (resp. primitive) in the sense above.
If $S$ is connected, the degree of $\xi$ is defined to be the degree of $\xi_s$ for any geometric point $s$.

Given $\xi\in \underline{\Pic}_{X/S}(S)$, the obstruction to lifting it to an element of $\Pic(X)$ is an element in the Brauer group $\mathrm{Br}(S)= H^2_{\et}(S,\GG_m)$.
In particular, given a polarized family of K3 surfaces, there exists an \'etale cover $T \rightarrow S$ such that $f_T: X_T \rightarrow T$ is a projective morphism of schemes.

We define the groupoid-valued functor $\M_{2d}$ on schemes by
\begin{align*}
\M_{2d}(S) = \{f: X \rightarrow S, \xi \in \underline{\Pic}_{X/S}(S)| &X \textrm{ is a family of K3 surfaces over }S,\\
&\text{with primitive quasipolarization }\xi, \deg \xi = 2d\}.
\end{align*}
We similarly define the functor $\M_{2d}^{\circ}$ of primitively polarized K3 surfaces; since ampleness is an open condition, it is easy to see that $\M_{2d}^{\circ}$ is an open subfunctor of $\M_{2d}$.
We will use the notation $\M_{2d,R}$ and $\M^{\circ}_{2d,R}$ to denote the restriction to schemes over $\Spec R$.

\begin{prop}
These functors are Deligne--Mumford stacks of finite type over $\Spec(\ZZ)$.
The stack
$\M_{2d,\ZZ[1/2d]}$ is smooth over $\Spec \ZZ[1/2d]$.
\end{prop}
This proposition is well-known, but since we can only find partial statements in the literature (\cite{olsson-k3} over $\QQ$ and \cite{rizov-k3} for the polarized case), we give a brief argument here.
Note that these stacks will not be separated because of the possibility of flops of $-2$-curves in one-parameter families of quasipolarized K3 surfaces.

\begin{proof}
Since the locus where a quasipolarization is primitive is open, it suffices to show representability without this condition.  

Given an algebraically closed field $k$, a double-point K3 surface $Y$ over $k$ is a projective surface with at worst rational double point singularities
whose minimal resolution is a (smooth) K3 surface over $k$.  Given a base scheme $S$, we define a family of double-point K3 surfaces, equipped with a polarization of degree $d$, over $S$ as before.  

We first show that, given a family of surfaces $g: Y \rightarrow B$ with isolated singularities, the property of having at-worst rational double point singularities is an open condition.  
Indeed, observe that, using upper semicontinuity of higher direct images, this is easy for a family which admits a simultaneous resolution in the sense of Artin--Brieskorn \cite{artin-brieskorn}.  In general, it suffices to show that the locus
is constructible and preserved under generization.  Constructibility
follows by applying Theorem $1$ of \cite{artin-brieskorn} to take a base change $B' \rightarrow B$ for which there exists a simultaneous resolution and using openness there.
 For generization, it suffices to check this at the complete local ring
for the versal deformation of a rational double point singularity; since rational double
points are unobstructed, Theorem $3$ of \cite{artin-brieskorn} shows that the rational double point locus is
open after taking a finite surjective base change.  This implies the
locus is open before the base change as well.

Furthermore, given a double-point K3 surface over an algebraically closed field $k$, with polarization $L$, the line bundle $L^{\otimes 3}$ is very ample, by Theorem $8.3$ of \cite{saint-donat} and the main theorem of \cite{terakawa} for $p=2$.
We can use Hilbert scheme arguments to show that the moduli functor $\M^{\doubp}_{2d}$ of polarized degree $2d$ double-point K3 surfaces is representable by a Deligne--Mumford stack of finite type.

To show representability of $\M_{2d}$, it suffices to construct a morphism $\pi: \M_{2d} \rightarrow \M^{\doubp}_{2d}$, representable in algebraic spaces.
Given a family of double-point K3 surfaces $g: Y \rightarrow S$, a simultaneous resolution of $g$ is a map of algebraic spaces $h: X \rightarrow Y$ such that the composition $g\circ h: X \rightarrow S$ is a family of K3 surfaces.  
We define the stack $\F_{2d}$ parametrizing families (over $S$) of double-point K3 surfaces $Y \rightarrow S$, equipped with a polarization of degree $d$ and a simultaneous resolution $h: X \rightarrow Y$.  Theorem $1$ of \cite{artin-brieskorn} shows that the natural morphism $\F_{2d} \rightarrow \M^{\doubp}_{2d}$ is representable by algebraic spaces.

Furthermore, since the pullback of the polarization on $Y$ defines a quasipolarization on $X$, we have a map $\F_{2d} \rightarrow \M_{2d}$.  
Using sections of a sufficiently high power of the quasipolarization ($3$ is enough), we can produce a family of polarized double-point K3 surfaces from a quasipolarized family of smooth K3 surfaces.  Using \cite{artin-sing} to show the quasipolarization descends, this provides an inverse map $\M_{2d} \rightarrow \F_{2d}$, so $\M_{2d}$ is also a Deligne--Mumford stack.

%We can construct an inverse as follows.
%Given $f: X \rightarrow S$ with quasipolarization $\xi$, let us assume that $\xi$ arises from a line bundle $L \in \Pic(X)$ which is nef and big on geometric fibers of $f$.  By Theorem $6.1$ of \cite{saint-donat}, sections of $L^{\otimes 3}$ define a morphism $X \rightarrow \mathrm{Proj}_{S}(\bigoplus 
%f_*(L^{\otimes 3k}))$ whose image $Y$ is a family of double-point K3 surfaces over $S$.  If $h: X \rightarrow Y$ denotes this contraction, it follows from Theorem $4$ of \cite{artin-sing} that $L = h^*h_*L$.  Indeed,  it is shown there that $h_*L$ is invertible and the higher direct images vanish. 
%Therefore the triple $(Y, h_*L, X)$ is a family of double-point K3 surfaces over $S$, with a polarization of degree $2d$, and a simultaneous resolution.  It is straightforward to extend this argument to the case where $\xi$ does not arise from a line bundle $L \in \Pic(X)$ by passing to an \'etale cover.  This shows that $\F_{2d} = \M_{2d}$ so that the latter is representable by a Deligne--Mumford stack.

For smoothness, given a quasipolarized K3 surface $(X,L)$ over a field $k$, if $2d$ is invertible in $k$, then the image of $c_{1}(L)$ in $H^1(X,\Omega^1_{X})$ is nonzero.  Therefore, by \cite{deligne-k3}, the versal deformation space is a hypersurface inside $\Spf W[[t_1,\dots, t_{20}]]$, smooth over $\Spf W$.
\end{proof}

\subsection{Clifford algebras}\label{cliffordsection}

In this section, we recall basic definitions for Clifford algebras and the Spin group.

\begin{defn}
Given a commutative ring $R$ and  a free $R$-module $M$ of finite rank equipped with a quadratic form $q: M \rightarrow R$,
the Clifford algebra 
$$\Cl(M,q) = TM/\langle m\otimes m = q(m)\rangle$$
is the quotient of the tensor algebra of $M$ by the two-sided ideal generated by the relation $q(m) = m \otimes m$ for all $m \in M$.
It is a $\ZZ/2$-graded, free module over $R$ with rank $2^{\mathrm{rank}(M)}$.
The even Clifford algebra $\Cl_+(M,q)$ is the even-graded component; it is also a free module of rank $2^{\mathrm{rank}(M)-1}$.
\end{defn}
In the examples we are interested in, the quadratic form will always be of the form $q(m) = \psi(m,m)$ for a symmetric bilinear form $\psi$ on $M$ and we
will use the notation $\Cl_+(M,\psi)$.
If the quadratic form is clear from context, we will typically suppress it from notation.  If we have an ordered basis of $M$, indexed
by a set $S$, then there is an associated basis of $\Cl_+(M)$ indexed by subsets $T$ of $S$ with even cardinality (obtained by multiplying the basis elements in the subset).
Note also that there exists a unique algebra anti-automorphism $\iota$ on $\Cl_+(M)$ that fixes $M$

We state now a few examples that we will use later.

\begin{example}\label{k3example}

Consider the integral lattice of rank $22$ 
\begin{equation}\label{k3lattice}
\LL = U^{\oplus 3} \oplus E_8^{\oplus 2}
\end{equation}
where $U$ is the hyperbolic lattice of rank $2$ and $E_8$ denotes the $E_8$-lattice.  If $e,f$ are standard basis elements of the first copy of the hyperbolic lattice (so $\langle e,e \rangle = \langle f,f \rangle = 0$ and 
  $\langle e,f\rangle = 1$), fix
$v_d = e- df$ and take its orthogonal complement
$$L_{2d} = (e - df)^{\perp} \subset \LL.$$
It is an indefinite lattice of rank $21$ with bilinear form $\psi$ of signature $(19,2)$.  In Section \ref{Hodgepositivity} we will use the associated even Clifford algebra $\Cl_+(L_{2d},\psi)$ in the context of the Kuga--Satake construction.  Similarly, if we tensor with $\QQ$, we have the quadratic space $V_{2d}$ and its even Clifford algebra $\Cl_+(V_{2d},\psi) = \Cl_+(L_{2d},\psi)\otimes\QQ$.    
\end{example}

\begin{example}
Given any K3 surface $X$ over an algebraically closed field $k$ with quasipolarization $L$, and a prime $l \neq\mathrm{char}(k)$, we can consider $H^2_{\et}(X, \ZZ_l(1))$
with the pairing given by the negative Poincar\'e pairing
$$\psi_{\et}: H^{2}_{\et}(X, \ZZ_l(1)) \otimes H^2_{\et}(X, \ZZ_l(1) \rightarrow H^4_{\et}(X, \ZZ_l(2)) = \ZZ_l$$
and the primitive cohomology given by taking the orthogonal complement of $c_1(L)$
$$P^2_{\et}(X, \ZZ_l(1)) = \langle c_1(L) \rangle^{\perp} \subset H^2_{\et}(X, \ZZ_l(1)).$$ 
If $L$ is primitive, then by lifting $(X,L)$ to characteristic zero and applying proper and smooth base change, we see that
$$\Cl_+(P^2_{\et}(X,\ZZ_l(1)), \psi_{\et}) \cong \Cl_+(L_{2d}) \otimes \ZZ_l$$
as $\ZZ_l$-modules.

Since $\psi_{\et}$ is compatible in families, the Clifford construction behaves well in families.
More precisely, given any family of K3 surfaces $f: X \rightarrow S$ with primitive quasipolarization $\xi$ of degree $2d$, and a prime $l$ invertible on $S$,
we can consider the lisse \'etale sheaf $R^2_{\et}f_*(\ZZ_l(1))$ on $S$, equipped with the pairing $\psi_{\et}$.
The quasipolarization $\xi$ defines a trivial rank $1$ subsheaf, and we can again take its orthogonal complement (with respect to $\psi_{\et}$):
$$\Pet(\ZZ_l(1)) := \langle c_{1}(\xi)\rangle^{\perp}\subset R^2_{\et}f_{*}(\ZZ_l(1))$$
and its associated Clifford algebra
$$\Cl_+(\Pet(\ZZ_l(1))),$$
a lisse \'etale sheaf on $S$ whose restriction to any geometric point agrees with the construction from the previous paragraph.
\end{example}

\begin{example}\label{clifford-dR}
We have de Rham versions of the previous example.  Given $(X,L)$ over $k$, such that $\charr k \nodivide 2d$, we have the de Rham cohomology group
$H^{2}_{\DR}(X)$ which is a $22$-dimensional $k$-vector space, equipped with
the descending Hodge filtration
$$0 \subset F^2 \subset F^1 \subset F^0 = H^{2}_{\DR}(X)$$
such that $(F^2)^{\perp} = F^1$.  It follows from \cite{deligne-k3} that the Hodge-to-de-Rham spectral sequence degenerates at $E_1$ and the dimensions of the associated graded pieces are $1$, $20$, and $1$ respectively.

As before, we consider the orthogonal complement of $c_1(L)$ with respect to the Poincar\'e pairing,
$$P^2_{\DR}(X) = \langle c_1(L)\rangle^{\perp} \subset H^{2}_{\DR}(X)$$
which inherits a three-step Hodge filtration with graded pieces of dimensions $1$, $19$, and $1$, and nondegenerate negative Poincar\'e pairing
$\psi_{\DR}$.
If we use $\{1\}$ to denote shifting filtration degree down by $1$, then the pairing defines a map of filtered vector spaces
$$\psi_{\DR}: P^2_{\DR}(X)\{1\} \otimes P^2_{\DR}(X)\{1\} \rightarrow k$$
where $k$ has trivial filtration concentrated in degree $0$.
The even tensor algebra on $P^2_{\DR}(X)\{1\}$ is equipped with a natural descending filtration induced from the Hodge filtration; therefore, we also have a natural descending filtration on the associated even Clifford algebra
$$\Cl_+(P^2_{\DR}(X)\{1\}),$$ obtained by taking the quotient filtration.
By choosing a filtered basis, it is easy to see that, since $P^2_{\DR}(X)\{1\}$ has one-dimensional isotropic $F^1$,
the filtration $\Fil$ on the Clifford algebra
has nonzero graded pieces only in degrees $1$, $0$ and $-1$,
and that $\Fil^1$
is spanned by elements of the form
$$\omega \cdot \prod \gamma_i$$
with $\omega \in F^1(P^2_{\DR}(X)\{1\})$
and $\gamma_i \in F^0(P^2_{\DR}(X)\{1\})$.

The same statements apply in families.  Suppose we are given a family of K3 surfaces
$f: X \rightarrow B$ with $2d$ invertible on $B$, such that $X$ is a scheme, equipped with a line bundle $L$ that gives a relative polarization.
Relative de Rham cohomology $R^2f_*(\Omega^*_{X/B})$ is a locally free sheaf of rank $22$.  It is equipped with the
descending Hodge filtration; the steps of this filtration are locally free subsheaves $\calF^i$ which are locally direct summands (this follows from the degeneration proven in \cite{deligne-k3}).  Again, if we take the orthogonal complement to $c_1(L)$,
we have an orthogonal splitting
$$R^2f_*(\Omega^*_{X/B}) = P^2_{\DR}(f) \oplus \calO_B\cdot c_1(L).$$
The locally free sheaf $P^2_{\DR}(f)\{1\}$ has a nondegenerate pairing and Hodge filtration (which we also denote by $\calF^i$), obtained by restriction.  The associated
Clifford algebra $$\Cl_+(P^2_{\DR}(f)\{1\})$$
is a locally free sheaf on $B$.  As in the last paragraph, it inherits a descending filtration
$\Fil^k$ which restricts on geometric points to the filtration given there.  By choosing (locally on $B$) a filtered basis of $P^2_{\DR}(f)$, we can see that the subsheaves of the filtration $\Fil^k$ are locally free and locally direct summands.
\end{example}

Finally, we recall here the definition of the Clifford and Spin groups.
\begin{defn}
The algebraic group $\mathrm{CSpin}(V_{2d})$ over $\QQ$ is given by
$$\mathrm{CSpin}(V_{2d}) = \{ g \in \Cl_+(V_{2d})^{*}| gV_{2d}g^{-1} \subset V_{2d}\}.$$
The adjoint action of $\mathrm{CSpin}(V_{2d})$ on $V_{2d}$ defines a map of algebraic groups 
$$\mathrm{ad}: \mathrm{CSpin}(V_{2d}) \rightarrow \mathrm{SO}(V_{2d}).$$
The Spin group $\mathrm{Spin}(V_{2d})$ is defined to be the kernel of the Norm map
$$\mathrm{Nm}: \mathrm{CSpin}(V_{2d}) \rightarrow \GG_m$$
given by $g \mapsto \iota(g)g\in \GG_m$.
\end{defn}

\subsection{Level structures}

It will be useful later to consider moduli of K3 surfaces with level structure determined by finite index subgroups of the groups $\mathrm{SO}$ and $\mathrm{CSpin}$.  The material in this section follows the discussion in \cite{andre, rizov-k3}.
We fix a degree $2d$ and an integer $n \geq 3$.  In practice, we will only work with level $n=4$.

Let $\ZZhat$ denote the profinite completion of $\ZZ$ and $\mathbb{A}_f = \ZZhat \otimes \QQ$ denote the ring of finite adeles.
Let $$\mathrm{CSpin}(L_{2d}) = \mathrm{CSpin}(V_{2d})(\mathbb{A}_f) \cap \Cl_{+}(L_{2d}\otimes \ZZhat)^*$$
and
let $$\KK_n^{\mathrm{sp}} \subset \mathrm{CSpin}(L_{2d})$$
be the open subgroup consisting of elements $\equiv 1 \bmod n$.
Finally, we set $$\KK_{n}^{\mathrm{ad}} \subset \mathrm{SO}(L_{2d}\otimes \ZZhat)$$ to be its image under the adjoint map.  It is proven in (\cite{andre},4.3) that this is an open subgroup of finite index.   
By construction, we can write $\KK_{n}^{\mathrm{ad}} = \prod_{p} \KK_{n,p}$ with respect to the decomposition $\mathrm{SO}(L_{2d}\otimes \ZZhat) = \prod_{p}
\mathrm{SO}(L_{2d}\otimes \ZZ_p)$.

Let $T$ be the set of primes dividing $2dn$, and let $\ZZ_T = \prod_{p \in T} \ZZ_p$.  Again, by (\cite{andre},4.3), 
this set includes all primes for which $\KK_{n,p}$ is a proper subgroup.
Given a family of K3 surfaces $f: X \rightarrow B$, with a primitive quasipolarization $\xi$ of degree $2d$, such that $B$ is connected and $2dn$ is invertible on $B$, we define a spin level $n$ structure as follows.

Fix a geometric base point $\bbar \rightarrow B$, and let
$P_{\et}^2(X_{\bbar},\ZZ_T(1))$ denote the primitive cohomology of the geometric fiber with coefficients in $\ZZ_T(1)$, defined as before by taking the orthogonal complement
of $c_1(\xi)$ with respect to the Poincar\'e pairing; it carries an action of $\pi_1^{\et}(B,\bbar)$.
%reference for statement of primes?  andre again?

Since $\xi$ is primitive, by lifting $(X_{\bbar},\xi)$ to characteristic zero, one can see that $P_{\et}^2(X_{\bbar},\ZZ_T(1))$ with the (negative) Poincar\'e pairing is isomorphic 
to $L_{2d}\otimes  \ZZ_T$.  Let $$\mathrm{Isom}( L_{2d}\otimes \ZZ_T, P_{\et}^2(X_{\bbar},\ZZ_T(1)))$$ denote the (nonempty) set of isometries between these two spaces, 
equipped with an action of $\prod_{p\in T} \KK_{n,p}$ via the left factor and an action of $\pi_1^{\et}(B,\bbar)$ via the right factor.

\begin{defn} A spin level $n$ structure on $(X,B,f, \xi)$ is an element of
$$\left(\prod_{p \in T} \KK_{n,p} \backslash \mathrm{Isom}( L_{2d}\otimes \ZZ_T, P_{\et}^2(X_{\bbar},\ZZ_T(1)))\right)^{\pi_1^{et}}.$$
\end{defn}
This definition is independent of the choice of base point; it extends to disconnected bases by working on each connected component separately.

We define the moduli functor $\M_{2d, n}$ of primitively quasipolarized K3 surfaces of degree $2d$, equipped with a spin structure of level $n$ in the obvious way.
\begin{prop}
$\M_{2d, n}$ is a smooth algebraic space over $\ZZ[1/2dn]$.  The forgetful map
$$\pi: \M_{2d,n} \rightarrow \M_{2d, \ZZ[1/2dn]}$$
is finite and \'etale.
\end{prop}
\begin{proof}
Given a connected scheme $B$ with geometric base point $\bbar$ and a map $B \rightarrow \M_{2d,\ZZ[1/2dn]}$, the fiber product
$$B' = \M_{2d,n} \times_{\M_{2d,\ZZ[1/2dn]}} B$$
is precisely the finite \'etale cover corresponding to the finite $\pi_1^{\et}(B,\bbar)$-set
$$\prod_{p \in T} \KK_{n,p} \backslash \mathrm{Isom}( L_{2d}\otimes \ZZ_T, P^{2}_{\et}(X_{\bbar},\ZZ_T(1))).$$
This shows the second claim, as well as smoothness and representability as a Deligne--Mumford stack.  It remains to show that there are no nontrivial automorphisms at any geometric point.

Suppose we have a quasipolarized K3 surface $(X,L)$ over an algebraically closed field $k$, and a finite order automorphism $\sigma$ of the pair that acts trivially on
$P_{\et}^2(X, \ZZ_T(1))\otimes \ZZ/n\ZZ$.  Since $n \geq 3$, the eigenvalues of $\sigma^*$ on 
$P_{\et}^2(X, \ZZ_T(1))$ are roots of unity which are $1 \bmod n$, which can only happen if they are equal to $1$.  Since $\sigma$ is finite order, it is semisimple, so acts trivially on $P_{\et}^2(X, \ZZ_T(1))$ and thus $H^2_{\et}(X, \ZZ_T(1))$.  Therefore, by \cite{rizov-k3}, 3.3.2, $\sigma$ must be be trivial.
\end{proof}

\section{Picard lattices in proper families}\label{sectionpicardjumping}

In this section, we will work over $\CC$ and study the moduli space $\M_{2d,\CC}$.

Let $\pi: \calX \rightarrow \M_{2d,\CC}$ be the universal family, and let
$$\lambda = \pi_*(\omega_{\calX/\M_{2d,\CC}})$$
denote the Hodge bundle on the moduli space.

Consider a collection of pairwise non-isomorphic rank $2$ lattices of the form:
$$\Lambda_{k} =  \left( \begin{array}{cc}
2d & a_{k}  \\
a_{k} & 2b_{k}  \end{array} \right),$$
for $k \in \ZZ^+$
with $\disc \Lambda_k < 0$.  
Let $D_{\Lambda_{k}} \subset \M_{2d,\CC}$ be the locus of quasipolarized K3 surfaces $(X,L)$ for which there exists an embedding of lattices
$$\Lambda_{k} \hookrightarrow \Pic(X)$$
that sends the first basis vector of $\Lambda_k$ to $L$.  This defines a divisor on $\M_{2d,\CC}$.

The main result of this section is the following:

\begin{thm}\label{Picardjumping}
There exists a Cartier divisor $D$ supported on a finite union of these divisors
$$\bigcup_{i=1}^{m} D_{\Lambda_{k_{i}}}$$
such that
$$\lambda^{\otimes a} = \calO(D) \in \Pic(\M_{2d,\CC})$$
for some $a > 0$.
\end{thm}

The Hodge bundle has positive degree on any non-isotrivial family of K3 surfaces over a proper curve.   Therefore, we have the following
corollary which is a refinement of \cite{borcherdsludmil}:

\begin{cor}\label{elliptic-case}
Let $f: X \rightarrow C$ be a non-isotrivial family of quasipolarized K3 surfaces of degree $2d$ over a proper curve $C$.
There exists a point $t \in C$ and a lattice $\Lambda_{k}$ in our collection 
such that we have an embedding of lattices:
$$\Lambda_{k} \subset \Pic(X_t).$$
In particular,
any nonisotrivial family of K3 surfaces over a proper curve $C$
contains an elliptic K3 surface.
\end{cor}

To prove the second claim, we apply the first claim to the collection of lattices
$$\Lambda_{k} =  \left( \begin{array}{cc}
2d & k  \\
k & 0  \end{array} \right).$$
A K3 surface $X$ has the structure of an elliptic fibration if and only if there exists $L' \in \Pic(X)$ with self-intersection zero.

\begin{remark}
For elliptic lattices, it is possible to prove Corollary \ref{elliptic-case} directly along the lines of \cite{oguiso} using a density criterion
of Green on Noether-Lefschetz loci and the fact that rational isotropic vectors in $\LL\otimes \RR$ are dense in the space of real isotropic vectors.  However, we need the more precise ampleness claim of Theorem \ref{Picardjumping} in order to move to characteristic $p$ later on.
\end{remark}

The proof of Theorem \ref{Picardjumping} is an application of Borcherds' construction of automorphic forms for $O(2,n)$ \cite{borch1, borch2}.  We will recall this work in the first two subsections and explain how to apply it to our setting.

\subsection{Vector-valued modular forms}

For a more detailed discussion of the material in this subsection and the next, we refer the reader to the original papers \cite{borch1, borch2}.

We first recall standard definitions regarding modular forms of 
half-integral weight.  In order to make sense of the modular transformation 
law with half-integer exponents, 
a double cover of the standard modular group $SL_{2}(\ZZ)$ is required. 

\begin{defn}
The metaplectic group $Mp_{2}(\ZZ)$ consists of pairs
$$\left(\left(\begin{array}{cc} a & b\\
c & d\end{array}\right), \phi(\tau) = \pm\sqrt{c\tau+d}\right)$$
where $\left(\begin{array}{cc}a & b \\ c & d\end{array}\right) \in SL_{2}(\ZZ)$
and $\phi(\tau)$ is a choice of square root of the function $c\tau+d$ on the upper-half plane $\mathcal{H}$.
The group structure is defined by the product 
$$\left(A_{1},\phi_{1}(\tau)\right)\cdot\left(A_{2},\phi_{2}(\tau)\right) 
= \left(A_{1}A_{2}, \phi_{1}(A_{2}\tau)\phi_{2}(\tau)\right).$$  
\end{defn}
Here, we write $A\tau$ for the usual action of $SL_{2}(\RR)$ on $\tau\in\mathcal{H}$.
The metaplectic group is generated by the two elements
$$T = \left(\left(\begin{array}{cc} 1 & 1\\ 0 & 1\end{array}\right), 1\right),
S =  \left(\left(\begin{array}{cc} 0 & -1\\ 1 & 0\end{array}\right), \sqrt{\tau}\right),$$
where $\sqrt{\tau}$ denotes the choice of square root with positive real part.

Let $\rho: Mp_{2}(\ZZ)\rightarrow \End_{\CC}(V)$ be a finite-dimensional representation of the metaplectic group such that $\rho$ factors 
through a finite quotient.  
\begin{defn}
Given $k \in \frac{1}{2}\ZZ$,
a modular form of weight $k$ and type $\rho$ is a holomorphic function
$$f: \mathcal{H} \rightarrow V$$
such that, for all 
$g =\left(A, \phi(\tau)\right)\in Mp_{2}(\ZZ)$,
we have
$$f(A\tau) = \phi(\tau)^{2k}\cdot \rho(g)(f(\tau)).$$  For $k\in \ZZ$ and $\rho$ trivial, this 
reduces to the usual transformation rule.
\end{defn}

If we fix an eigenbasis $\{v_{\gamma}\}$ for $V$ 
with respect to $T$, we can take the Fourier expansion 
of each component of $f$ at the cusp at infinity.  That is, we write
$$f(\tau) = \sum_{\gamma} \sum_{k\in \ZZ} c_{k,\gamma} q^{k/R} v_{\gamma} \in V$$ 
where 
$$q = e^{2\pi i \tau}$$
and
$R$ is the smallest positive integer for which $T^{R}\in \mathrm{Ker}(\rho)$.    
The function $f$ is holomorphic at infinity if $c_{k,r} = 0$
for $k < 0$.  The space $$\mathrm{Mod}(Mp_{2}(\ZZ), k,\rho)$$ of holomorphic modular forms of 
weight $k$ and type $\rho$ is finite-dimensional.  If $c_{0,\gamma} = 0$ for all $\gamma$, we say
that $f$ is a cusp form.

Given an integral lattice $M$ with
an even bilinear form $(,)$ of signature{\footnote{Notice the signature here (following \cite{borch1}) differs from the
conventions in Section $2$ where we take the negative Poincar\'e pairing to match \cite{andre}.  This leads to some extra signs, but is otherwise harmless.}}
 $(2,n)$, we associate to $M$ the following unitary representation of $Mp_{2}(\ZZ)$.
Let $$M^{\vee} \subset M \otimes \QQ$$
denote the dual lattice and $M^{\vee}/M$ the finite quotient.  
The pairing extends linearly
to a $\QQ$-valued pairing on $M^{\vee}$.  The functions 
$\frac{1}{2}(\gamma, \gamma)$ and $(\gamma,\delta)$ descend to $\QQ/\ZZ$-valued
functions on $M^{\vee}/M$.

We can define a representation $\rho_{M}$ of $Mp_{2}(\ZZ)$ on the group algebra
$\CC[M^{\vee}/M]$ as follows, in terms of the action of the generators $T$ and $S$ with respect
to the standard basis $v_{\gamma}$ for  $\gamma \in M^{\vee}/M$,

\begin{align*}
\rho_{M}(T)v_{\gamma} &= e^{2\pi i\frac{(\gamma,\gamma)}{2}} v_{\gamma}\ ,\\
\rho_{M}(S)v_{\gamma} &= \frac{\sqrt{i}^{n-2}}{\sqrt{|M^{\vee}/M|}}
\sum_{\delta} e^{-2\pi i (\gamma,\delta)} v_{\delta}\ .
\end{align*}

We will apply all this to $M = L_{2d}$, equipped with the
bilinear form
$$(\gamma, \delta) = -\psi(\gamma,\delta);$$
we take the negative of the bilinear form considered in 
Example \ref{k3example} to match the signature conventions of \cite{borch1} and this section.
In this case, we have
$$M^{\vee}/M = \ZZ/2d\ZZ.$$
For the representation, we will take the dual representation $$\rho^* = \rho_{L_{2d}}^*.$$
We take the dual to match conventions in \cite{borch1}.  It follows from McGraw \cite{mcgraw} that the complex vector space
$\mathrm{Mod}(Mp_{2}(\ZZ), k,\rho^*)$ has a rational structure $\mathrm{Mod}(Mp_{2}(\ZZ), k,\rho^*)_{\QQ}$ given by modular forms with rational coefficients.

\subsection{Recap of Borcherds' work}

Recall the period domain
$$\Omega^{\pm} = \{\omega \in L_{2d}\otimes\CC| \psi(\omega, \omega) = 0, -\psi(\omega, \overline{\omega}) > 0 \}$$
and consider the arithmetic subgroup
$$\Gamma = \mathrm{Aut}(\LL, v_{2d})$$
of $O(V_{2d})$ acting on $\Omega^{\pm}$.

The analytic orbifold quotient $$[\Omega^{\pm}/\Gamma]$$ naturally has the structure of a smooth algebraic Deligne--Mumford stack by \cite{baily-borel}, and the period map defines a morphism
$$j: \M_{2d,\CC} \rightarrow [\Omega^{\pm}/\Gamma]$$
which is an open immersion on the polarized locus.

For every $$n\in\QQ^{<0},\ \ \gamma \in L_{2d}^{\vee}/L_{2d} = \ZZ/2d\ZZ,$$ we associate a divisor 
$y_{n,\gamma}$ on $[\Omega^{\pm}/\Gamma]$
as follows.
Given an element $v \in L_{2d}^{\vee}$, there is an associated hyperplane
$$v^{\perp} = \left\{ \omega \in\Omega^{\pm}\ |\ \psi(\omega,v) = 0\right\}\subset \Omega^{\pm}.$$
Both $\psi(v,v)$ and the residue class $v \bmod L_{2d}$ are
invariant under the action of $\Gamma$.  Therefore, if we fix $n$ and $\gamma$ as above, the set of $v \in L_{2d}^{\vee}$ with 
$$\psi(v,v) = -n,\ \  v \equiv \gamma \bmod L_{2d}$$
is also $\Gamma$-invariant.  
The union over the set of the associated hyperplanes
$$\sum_{\shortstack{$\psi(v,v) = -n$ \\$v \equiv \gamma \bmod L_{2d}$}} v^{\perp}$$
is $\Gamma$-invariant and descends to an algebraic divisor
$$y_{n,\gamma}  = \left(\sum_{\psi(v,v)=-n,\ v\equiv \gamma \bmod L_{2d}} v^{\perp}\right)/\Gamma.$$
The $y_{n,\gamma}$ are the {\em Heegner divisors} of $[\Omega^{\pm}/\Gamma]$; let $[y_{n,\gamma}]$ denote the associated line bundle.
Because of the symmetry $v^{\perp} = (-v)^{\perp}$,
there is a redundancy $y_{n,\gamma} = y_{n,-\gamma}$ in our notation.
It follows from these definitions that $y_{n,\gamma} = 0$ unless
\begin{equation}\label{vanishing}
-n \equiv \frac{1}{4d} \gamma^2 \bmod \ZZ.
\end{equation}
%and $y_{n,\gamma}$ is multiplicity $2$ everywhere if $2\gamma \equiv 0 \bmod L_{2d}$.  

In the degenerate case where $n = 0$, we have the following prescription.
The line bundle $\mathcal{O}(-1)$ on $\Omega^{\pm} \subset \PP(L_{2d}\otimes \CC)$
admits a natural $\Gamma$ action and therefore descends to a line bundle $K$ on 
$[\Omega^{\pm}/\Gamma]$.   If $n = 0$ and $\gamma = 0$, we set
$$[y_{0,0}] = K^{*}.$$
If $n=0$ and $\gamma \neq 0$, we set $y_{n,\gamma} = 0$.

Given a rank $2$ lattice
$$\Lambda =  \left( \begin{array}{cc}
2d & a  \\
a & 2b  \end{array} \right),$$
we can associate the discrete invariants
$$n = \frac{\mathrm{disc}\Lambda}{4d} = b - \frac{a^2}{4d};\ \ \gamma \equiv a \bmod 2d$$
and the associated Heegner divisor $y_{n,\gamma}$.
It is clear from definitions that
$j^{*}(y_{n,k})$ is a divisor with support $D_{\Lambda}$, but possibly with multiplicities.
Similarly, in the degenerate case, we have 
$$j^*([y_{0,0}]) = -\lambda.$$

We can place the Heegner divisors 
in a formal power series $\Phi_{2d}(q)$ with coefficients in
$\mathrm{Pic}([\Omega^{\pm}/\Gamma])\otimes \QQ[L_{2d}^{\vee}/L_{2d}]$. We can define the generating function 
$$\Phi(q) = \sum_{n\in \QQ^{\geq 0}}\sum_{\gamma \in \ZZ/2d\ZZ} [y_{-n,\gamma}] q^{n} v_{\gamma} \in \Pic([\Omega^{\pm}/\Gamma])[[q^{1/4d}]]\otimes \QQ[L_{2d}^{\vee}/L_{2d}].$$

The following proposition is Theorem $4.5$ of \cite{borch2} together with the refinement of \cite{mcgraw}:

\begin{prop}\label{borcherds-mainresult}
The generating function $\Phi_{2d}(q)$ is an element of
$$\Pic([\Omega^{\pm}/\Gamma])\otimes \mathrm{Mod}(Mp_{2}(\ZZ), 21/2, \rho^*)_{\QQ}.$$
\end{prop}

In particular, given any linear functional
$$\lambda: \mathrm{Pic}([\Omega^{\pm}/\Gamma]))\otimes\QQ \rightarrow \QQ,$$
the contraction $\lambda(\Phi_{2d}(q))$ is the Fourier expansion of a vector-valued
modular form of weight 
$21/2$ and type $\rho^{\ast}$, so we have a map
$$\beta: \left(\mathrm{Pic}([\Omega^{\pm}/\Gamma]))\otimes\QQ\right)^* \rightarrow  \mathrm{Mod}(Mp_{2}(\ZZ), 21/2, \rho^*)_{\QQ}.$$

\subsection{Proof of Theorem \ref{Picardjumping}}

With this background in place, it is easy to explain the proof.  
It follows from equation \eqref{vanishing} that every modular form in the image of $\beta$ has the following vanishing property:
\begin{equation}\label{vanishing2}
c_{n,\gamma} = 0 \mbox{  unless  } n \equiv \frac{1}{4d}\gamma^2 \bmod \ZZ;\ \ c_{0,\gamma} = 0\mbox{ if }\gamma\ne 0.
\end{equation}
Let $$\mathrm{VMod}(Mp_{2}(\ZZ), 21/2, \rho^*)_{\QQ}$$ denote the rational subspace of modular forms satisfying this vanishing condition.
The key point is the following elementary lemma about vector-valued modular forms in this subspace.

\begin{lem}\label{cuspforms}
Suppose we have 
$$f(q) = \sum_{n\in \QQ_{\geq 0},\gamma} c_{n,\gamma} q^{n} v_{\gamma} \in  \mathrm{VMod}(Mp_{2}(\ZZ), 21/2, \rho^*)$$
such that
$$c_{0,0} \ne 0.$$
Then, for each $\gamma \in \ZZ/2d\ZZ$, 
$$c_{n,\gamma} \ne 0$$ for all $n\in \QQ_{\geq 0}$ sufficiently large
such that
$$n \equiv \frac{1}{4d}\gamma^2 \bmod \ZZ.$$
\end{lem}
\begin{proof}

Let $\theta_{2d}(q) \in \mathrm{Mod}(Mp_{2}(\ZZ), 1/2, \rho^*)$ denote the Siegel theta function, the vector-valued modular form of weight $1/2$ whose Fourier expansion is
$$\theta_{2d}(q) = \sum_{\gamma =0}^{2d-1}\sum_{s} q^{\frac{(2ds + \gamma)^{2}}{4d}}v_{\gamma} \in \mathrm{Mod}(Mp_{2}(\ZZ), 1/2, \rho^*).$$
If we take the Eisenstein series 
$$E_{10}(q) = 1 - 264\sum_{n\in \ZZ_{+}} \sigma_{9}(n)q^{n},$$
then
$$E_{10}(q)\cdot \theta_{2d}(q) \in \mathrm{VMod}(Mp_{2}(\ZZ), 21/2, \rho^*)_{\QQ},$$
i.e., it satisfies the same vanishing conditions as $f$.

Furthermore, since its only nonzero constant term is when $\gamma = 0$, we see that
$$f(q) - c_{0,0}\cdot E_{10}(q)\cdot \theta_{2d}(q)$$
is a cusp form.  In particular, each component is a cusp form in the usual sense.

As is well-known (see Corollary $2.1.6$ of \cite{miyake} or Proposition $1.3.5$ of \cite{sarnak} for the half-integral weight case), there is a trivial bound on the growth of Fourier coefficients of cusp forms:
$$|c_{n,\gamma}| < n^{\mathrm{wt}/2+ \epsilon} = n^{21/4+\epsilon}$$ 
for each $\gamma$ and $n$ sufficiently large.

Given $n$ and $\gamma \in \ZZ/2d\ZZ$ such that 
$$n \equiv \frac{1}{4d}\gamma^2 \bmod \ZZ,$$
The corresponding coefficient of  $E_{10}(q)\cdot \theta_{2d}(q)$ is nonzero and has magnitude bounded from below by 
$264(n- d/4)^{9}$.  Therefore, this contribution to $f(q)$ dominates the cusp form term for $n>>0$, so $c_{n,\gamma}$ is nonzero.
\end{proof}

We now prove Theorem \ref{Picardjumping}.

\begin{proof}
Given $\Lambda_k$ in our collection of rank $2$ lattices, 
let $$y_{k} = y_{n_{k},\gamma_{k}}$$ denote the associated Heegner divisors.

It suffices to show that we have a linear dependence of Heegner divisors
$$a [y_{0,0}] + \sum_{i=1}^{m} c_{i} [y_{k_{i}}] = 0 \in \Pic([\Omega^{\pm}/\Gamma])$$
with $a > 0$.  Indeed, if we pull back this dependence via the period map, we get the claim of the proposition.

To see this, let 
$$H = \mathrm{Span}\{[y_{k}]\} \subset \Pic([\Omega^{\pm}/\Gamma])\otimes\QQ$$
denote the (finite-dimensional) linear span of our collection of Heegner divisors. Then we want to show that
$$[y_{0,0}] \in H.$$

We can rephrase the statement of Proposition \ref{borcherds-mainresult} as saying that we have a diagram:
\begin{equation*}
\bigoplus \QQ e_{n,\gamma} \stackrel{\alpha}{\longrightarrow} \mathrm{VMod}(Mp_{2}(\ZZ), 1/2, \rho^*)_{\QQ}^{\ast} \stackrel{\beta^*}{\longrightarrow} \Pic([\Omega^{\pm}/\Gamma])\otimes\QQ.
\end{equation*}
Here, the first map $\alpha$ sends
the basis vector
$e_{n,\gamma}$ to the linear functional
$$f(q) = \sum c_{n,\gamma}q^{n}v_{\gamma} \mapsto c_{n,\gamma},$$
and the second map $\beta^*$
is the dual of the map $\beta$ defined in the last section.
By construction, we have the composition 
$$\beta^*\circ\alpha: e_{n,\gamma} \mapsto [y_{n,\gamma}].$$

Since $\Lambda_{k}$ are pairwise non-isomorphic, we know that $n_{k} \rightarrow -\infty$ as $k \rightarrow \infty$.  Therefore, Lemma \ref{cuspforms} implies that given a modular form 
$$f(q) \in \mathrm{VMod}(Mp_{2}(\ZZ), 1/2, \rho^*) = \mathrm{VMod}(Mp_{2}(\ZZ), 1/2, \rho^*)^{**}$$
that vanishes on $\mathrm{Span}(\alpha(e_{n_{k},\gamma_{k}}))$,
it must vanish on $\alpha(e_{0,0})$ as well.  Therefore, we must have
$$\alpha(e_{0,0}) \in \mathrm{Span}(\alpha(e_{n_{k},\gamma_{k}}))$$
By applying $\beta^*$, we have that
$y_{0,0}\in H$.
\end{proof}

\section{Degenerations of supersingular K3 surfaces}\label{section-ss}

In this section, we study one-parameter degenerations of supersingular K3 surfaces.

Let $k$ be an algebraically closed field of characteristic $p$.  Let $K$ be a discrete valuation field
with residue field $k$, and let $\calO_K$ be its valuation ring, which we assume to have characteristic $p$ as well.
Set
$$\Delta = \Spec \calO_K.$$
Suppose we are given a supersingular K3 surface
$$f: X \rightarrow \Spec K$$
equipped with a very ample bundle $L$ of degree $2d$, such that $p >2d+4$.

The main result of this section is the following proposition.

\begin{thm}\label{ss-degeneration}
After possibly taking a finite, separable base change,
$f$ extends to a family of K3 surfaces
$$g: \calX \rightarrow \Delta.$$
so that $L$ extends to a quasipolarization $L'$.
\end{thm}

By a finite separable base change, we mean a finite extension of (discrete valuation) fields $K \rightarrow K'$, so that our family extends smoothly over $\Delta' = \Spec \calO_{K'}$.  Since we will be making a series of such extensions, we will abuse notation and write $\Delta$ at each step.

Supersingular K3 surfaces are closed with respect to specialization by Corollary $1.3$ of \cite{artin-ss}, so $X'$ has supersingular central fiber as well.

It seems that the principle that supersingular $K3$ surfaces should not degenerate was first applied by Rudakov--Zink--Shafarevich \cite{rsz}, who prove this for polarizations of degree $2$ and deduce Artin's conjecture in this case as a consequence.  Their argument can be extended using more recent techniques in birational geometry as follows:

\begin{enumerate}
\item After blowing up $X$ and making a separable base change, we construct a semistable model $\calY$ over $\Delta$, using a semistable reduction result \cite{saito} of T. Saito for iterated fibrations of curves.  This step is where the bound on $p$ is required.

\item We replace $\calY$ with $\calX$, a minimal model with trivial relative canonical bundle, by running semistable MMP for surfaces over a DVR, proved by Kawamata \cite{kawamata} in mixed and finite characteristic.  In particular, if we restrict to $\Spec K$, we recover the generic fiber $X$.  Using an argument of Corti \cite{corti} over $\CC$, the possible central fibers are classified by Kulikov, Pinkham-Persson \cite{kulikov, pinkham-persson}.

\item Finally, using $\QQ$-factoriality of $\calX$ and the argument of \cite{rsz}, supersingularity of the generic fiber forces $g$ to be smooth.

\end{enumerate}

The argument is somewhat complicated by requiring the extension of $L$ to be a quasipolarization.  In particular, we need to make sure that the steps in the theorems of Saito and Kawamata can be arranged to preserve $L$.   In the first case, this is straightforward.  For Kawamata's paper, this is a little subtle due to flipping contractions; fortunately, we can apply MMP \emph{with scaling}, a standard refinement of the usual approach, which allows us to follow $L$ through these rational maps.  We are grateful to A. Corti for a discussion of these issues.

In this section, we will work with Weil divisors considered up to linear equivalence.
It will also be convenient to work $\QQ$-divisors, i.e. rational linear combinations of Weil divisors, and consider these up to $\QQ$-linear-equivalence, which we denote by $\sim_{\QQ}$.  For the most part, we will work with $\QQ$-factorial schemes, so all $\QQ$-divisors will be $\QQ$-Cartier.

Since the Picard group $\Pic(X_{\overline{K}})$ is finitely generated, we can assume, by taking a finite separable base change, that every line bundle on $X_{\overline{K}}$ is defined over $K$.
%In the case of Saito's paper, this is proven in his arguments, although  it will probably be irritating to make that explicit.
%\Dcom{fuck you Saito!}

\subsection{Construction of a semistable model}\label{saito-section}

We first find a semistable model for $X$ after blowup and base extension.

We begin with the following lemma:
\begin{lem}\label{genericpencil}
There exists a pencil of curves on $X$ (defined over $K$) for the linear system $|L|$ associated to $L$
\begin{equation*}
\xymatrix{
Y \ar[r]^{\tau}\ar[d]^{\pi} & X\\
\PP^1_{K} & 
}
\end{equation*}
such that
$Y$ is smooth and the closed fibers of $\pi$ are irreducible, nodal curves of genus $g=d+1$.
\end{lem}
Since $\pi$ is a pencil of curves, $\tau$ is a proper birational map.

\begin{proof}

It suffices to pass to the algebraic closure $\overline{K}$ and prove the statement for a generic pencil, since we can then always find one defined over $K$.  For convenience, we suppress the subscript $\overline{K}$.

We first show that a generic pencil has smooth total space with at worst nodal fibers.
Since $L$ is very ample, we have an embedding
$$X \subset \PP^{g},$$
where we have used Riemann-Roch to determine
$h^0(L) = d+2 = g+ 1.$
Let $\PP^{\vee}$ denote the dual projective space, let
$\pr_{\calC}: \calC \rightarrow \PP^\vee$
denote the universal family of hyperplane sections of $X$,
and let
$$Z(X) = \{(x,H) | x \in X, H \in \PP^\vee, T_{x} X \subset H\} \subset \calC \subset \PP^g \times \PP^{\vee}$$
denote the locus of the incidence variety parametrizing singular points.
It is smooth of dimension $g -1$, since it is a projective bundle over $X$.
By definition, the dual variety $X^\vee$ is the (reduced) image of the projection
of $Z(X)$ to $\PP^{\vee}$:
$$Z(X)  \stackrel{\pr}{\longrightarrow} X^\vee \subset \PP^\vee.$$

We first show that the projection $\pr: Z(X) \rightarrow X^\vee$ is birational.
Consider the intersection number
$$\deg \pr^*(h)^{g-1} \cap [Z(X)] = \deg h^{g-1}\cap \pr_*[Z(X)] = 6d + 24,$$
where this calculation follows from the Pl\"{u}cker formula for the dual variety (\cite{holme}).
Since this is nonzero, we must have $X^{\vee}$ is a hypersurface and $\pr$ is generically finite.  
Since $$p > 2d+ 4 > d + 4,$$
$p$ does not divide $\deg(\pr)$,
so $\pr$ must be generically \'etale.
In fact, since $\pr$ is generically \'etale, it follows from Theorem $2.1$ of \cite{holme} that 
 $X$ is reflexive (meaning $Z(X^{\vee}) = Z(X)$).  Furthermore, by Corollary $2.2$ of \cite{holme},
 the generic fiber of $\pr$ is a linear space, so it must be a point and $\pr$ is birational.  
 Let $X^\vee_\circ$ denote the locus on $X^\vee$ over which $\pr$ is an isomorphism.
Given a hyperplane $[H] \in X^\vee_\circ$, it is shown in (\cite{holme}, Cor. $2.2$) that 
$\mathrm{Sing}(H \cap X)$ is nondegenerate, so must consist of a single node.

Suppose we have a line $\ell \subset \PP^\vee$ which intersects $X^\vee$ transversely at a finite number of points inside $X^\vee_\circ$.  It follows from a dimension count that the set of such lines is open.  Given a point $(x,[H]) \in \calC \times_{\PP^\vee}\ell$, its tangent space is the kernel of
$$T_{x,[H]}\calC \oplus T_{[H]}\ell \rightarrow T_{[H]}\PP^{\vee}.$$
Because $T_{[H]}X^\vee$ is contained in the image of $T_{x,[H]}\calC$, transversality implies this map is surjective, so the kernel is two-dimensional and  $\calC \times_{\PP^\vee}\ell$ is smooth.

We next show that a generic pencil will have irreducible fibers.
It suffices to show that the generic element of $X^\vee$ is a geometrically irreducible hyperplane section.

If not, we have a very ample divisor with a decomposition into smooth connected curves
$$C = C_1 + C_2$$
such that $C_1\cdot C_2 = 1$.
By the Hodge index theorem, $1 > (C_1\cdot C_1)(C_2\cdot C_2)$ so one of the curves (say $C_1$) is either rational or elliptic.
In the first case, $L\cdot C_1 = -1$ and in the second case $L\cdot C_1 = 1$; both of these contradict the fact that $L$ is very ample.

%We proceed by contradiction.
%Suppose we have a decomposition $L = L_1 + L_2 \in \Pic(X)$ such that a general hyperplane section in $X^\vee$ can be written as a union of reduced connected curves in $|L_1|$ and $|L_2|$.  
%If we set $2d_i = L_i\cdot L_i$, then the dimension of the linear system $|L_i|$ is $d_i + 1$, since they contain reduced connected curves;  in order for these to cover $X^{\vee}$, we must have
%$$d \leq (d_1+1) + (d_2 + 1),$$
%or equivalently $\delta = L_1\cdot L_2 \leq 2$.

%Given a reduced connected curve $C$ on a $K3$ surface, the
%exact sequence
%$$ H^0(X, \calO) \rightarrow H^0(X, \calO_C) \rightarrow H^1(X, \calO(-C)) \rightarrow 0$$
%implies that 
%$h^1(\calO(-C))=h^1(\calO(C)) = 0$.
\end{proof}

\begin{remark}
Over $\overline{K}$, the discriminant locus of a generic pencil is reduced.
Therefore, after taking a finite separable extension, we can choose our line over $K$ to have the property that
there exists a collection of $K$-points $D \subset \PP^1_K$, such that $Y$ is smooth on the complement of $D$.
\end{remark}

We fix $Y$ as in the statement of the proposition.  In what follows, let
$$Y_{\mathrm{sm}} = \pi^{-1}(\PP^1_K - D)$$
denote the locus of smooth fibers of the map $\pi$.

\begin{prop}\label{semistable}
After possibly replacing $Y$ with the base change by a finite separable extension of $K$,
there exists a birational map
$$\rho: Y' \rightarrow Y$$
such that $\rho$ is an isomorphism over the locus $Y_{\mathrm{sm}}$
and such that $Y'$ is the generic fiber of a semistable model
$$g: \calY \rightarrow \Delta.$$  Furthermore, there exists an ample $\QQ$-divisor $\calL$ on $\calY$ whose restriction 
to $Y'$ agrees with 
$\tau^*L$ on $\rho^{-1}(Y_{\mathrm{sm}})$,
up to $\QQ$-linear equivalence.
\end{prop}
\begin{proof}

Without the condition on the line bundle, this essentially follows from Theorems $1.3$ and $1.8$ in \cite{saito}.  
The basic idea of Saito's argument is to use theorems of \cite{dejong-oort} and \cite{mochizuki} for extending stable curves over the complement of an open subset of the base.  The condition on the prime $p > 2g+2= 2d+4$ is imposed to guarantee that we can extend the stable curve in codimension one.  Given an iterated fibration of curves, he applies this principle repeatedly.

In our case, we first extend the pair $(\PP^1_K, D)$ over $\Spec K$ (after possible base change) to 
a marked genus $0$ stable curve 
$$B \rightarrow \Delta.$$
The cited theorems in \cite{saito} then give an extension of $Y_{\mathrm{sm}}$ to a log smooth proper scheme
$\calY_0 \rightarrow B$ and, after a further log blowup away from $Y_{\mathrm{sm}}$, to a semistable relative surface
$$\calY \rightarrow B \rightarrow \Delta.$$  Furthermore, since $\calY_0 \rightarrow B$ is log smooth, it is a family of nodal curves, so there is a stabilization map
$(\calY_0)_K \rightarrow Y$.  If we combine this with the map $\calY_K \rightarrow (\calY_0)_K$, we have the birational map
$$\rho: \calY_K \rightarrow Y.$$

It remains to construct the $\QQ$-divisor $\calL$.  
First notice that if we consider the relative dualizing sheaf $\omega_{Y/\PP^{1}_{K}}$, we have an equality
$$\tau^*L = \omega_{Y/\PP^{1}_{K}}$$
when restricted to $Y_{\mathrm{sm}}$.  Indeed, they agree fiberwise by the adjunction formula, so differ by the pullback of an element of $\Pic(\PP^1_{K} - D) = 0$.

We first construct a relatively ample $\QQ$-divisor $\calL'$ for the map $\calY \rightarrow B$
that agrees with $\tau^*L$ on $Y_{\mathrm{sm}}$.
The scheme $\calY_0$ is constructed along with a relatively ample $\QQ$-divisor extending the relative dualizing sheaf (see the discussion on page 29 of
\cite{saito}).
Since $\calY \rightarrow \calY_0$ is the normalization of a blowup map, there exists a relatively ample line bundle for this map supported on the exceptional locus.  After twisting by some possibly fractional power of this bundle, we can construct $\calL'$.

Finally, choose an ample divisor $A'$ on $B$ that, when restricted to $\PP^{1}_K$, is supported on $D$; 
We then construct $\calL$ by twisting $\calL'$ by a sufficiently large multiple of the pullback of $A'$. Since $A'$ is trivial on $Y_{\mathrm{sm}}$, 
we have
$$\calL|_{\rho^{-1}(Y_{\mathrm{sm}})} \sim_{\QQ} \rho^*\tau^*L|_{\rho^{-1}(Y_{\mathrm{sm}})}.$$
\end{proof}

\subsection{Semistable MMP with scaling}

We now apply Kawamata's semistable MMP to write down a minimal model of
$\calY \rightarrow \Delta$.  We first give a brief overview of his results before applying them to our setting.

Let$$f: V \rightarrow \Delta$$ be a flat, projective morphism, with relative dimension $2$ and with $V$ normal.  We can define a Weil divisor $K_{V/\Delta}$, well-defined up to linear equivalence, associated to the rank $1$ reflexive sheaf 
$\iota_*(\omega_{V^{\circ}/\Delta})$ where $\iota: V^{\circ} \rightarrow V$ is the inclusion of the locus where $f$ is smooth.  

In Kawamata's paper, he imposes the following conditions (**) on the singularities of $V$:
\begin{enumerate}
\item $V$ is Cohen-Macaulay, $\QQ$-factorial, and regular away from finitely many points
\item $f$ has smooth generic fiber, 
\item the special fiber $f^*(s)$ is reduced,
\item $K_{V/\Delta}+f^*(s)$ is log terminal.
\end{enumerate}

We define a one-cycle on $V/\Delta$ to be an integral linear combination of irreducible, reduced one-dimensional closed subshemes $C \subset V$, contained in fibers of $f$.  Let $N_1(V/\Delta)$ denote the group of one-cycles modulo numerical equivalence and $N_1(V/\Delta)_{\RR} = N_1(V/\Delta)\otimes\RR$.  We take $\NE(V/\Delta)$ to be the closed convex cone inside generated by effective one-cycles on $V$.

Given a $K_{V/\Delta}$-negative extremal ray $R$ of the cone $\NE(V/\Delta)$, the contraction theorem asserts that there exists a projective surjective map
$$\pi_R:V \rightarrow Z$$ to a flat normal scheme $Z/\Delta$ such that
$\pi_R$ contracts a one-cycle $C$ if and only if $[C] \in R \subset \NE(V/\Delta)$. 

There are three possibilities for the map $\pi_R$:
\begin{enumerate}
\item $\pi_R$ is a fibration with $\dim Z < 3$
\item $\pi_R$ is birational, with divisorial exceptional locus.
\item $\pi_R$ is birational with small exceptional locus.
\end{enumerate}

%\begin{remark} Although one defines $\pi_R$ by choosing an ample bundle, it turns out to be independent of this choice.  
%\end{remark}

%To see this, it suffices to check this over the generic point of $\Delta$, since $Z$ is normal (A generically smooth proper surjective morphism is characterized by its fibers).   Here it follows from MMP for surfaces in characteristic $p$.

%are extremal rays over \eta Galois orbits of extremal rays over \etabar?

In the second case, $Z/\Delta$ again satisfies the conditions (**).  In the third case, there exists a flip $$\pi^+: V^+ \rightarrow Z,$$
i.e., a small birational contraction such that $K_{V^+/\Delta}$ is $\pi^+$-ample and $V^+$ satisfies conditions (**).  By choosing $K$-negative extremal rays and repeating this process, this algorithm eventually terminates.  The endpoint is either a fibration or a birational model satisfying (**) with $K_{V/\Delta}$ nef.

Suppose we are given a $\QQ$-divisor $H$ on $V$ such that $K_{V/\Delta} + H$ is nef.
The following lemma explains how to preserve this condition while applying the above package.
\begin{lem}\label{scaling}
Suppose that $K_{V/\Delta}$ is not nef. There exists an extremal ray $R$ and positive rational $\alpha \leq 1$ such that either $\pi_R$ is a fibration or, if $\tau: V \dasharrow V'$ denotes the modification associated to $R$ (either divisorial contraction or flip), then $K_{V'/\Delta} + \tau_*(\alpha H)$ is nef on $V'$, where $\tau_*$ denotes either pushforward or strict transform.
\end{lem}
\begin{proof}
The proof of this lemma is well-known and appears in many places, for example \cite{fujino}.
We sketch the argument.

Consider the set
$$S_H = \{t \in \QQ| K_{V/\Delta} + tH \textrm{ nef}\}.$$
Since $K_{V/\Delta}$ is not nef, $S_H$ is bounded below by $0$. 
Let
$\alpha = \mathrm{inf} S_H$.
It is easy to see that $\alpha$ is determined by $K_{V/\Delta}$-negative extremal rays:
$$\alpha = \mathrm{sup}_{i}\left\{\frac{-K_{V/\Delta}\cdot R_i}{H\cdot R_i} \right\}.$$

It follows from the cone theorem \cite{kawamata} that each $K$-negative extremal ray $R_i$ is generated by an irreducible curve $C_i$ for which $0 < -K_{V/\Delta}\cdot C_i < 4$.  Indeed, such a bound holds for extremal rays for a log surface $(S,D)$ in characteristic $p$ by Propositoin $2.9$ of \cite{tsunoda-miyanishi}; Kawamata shows that extremal rays for $V/\Delta$ are generated by extremal rays for irreducible components of fibers.  

%how does this bound work over the generic fiber?Need to include a term for size of Galois orbit.  replace R with R/G in the sup.  K.R and H.R still integral! actually maybe I only need the closed fiber? }

Although there are countably many rays, there are only finitely many values of the numerator in the above expression, so the supremum is achieved on some extremal ray $R$.

If $\pi_R$ is a divisorial contraction then, since $(K_{V/\Delta} + \alpha H)\cdot R = 0$, $(K_{V/\Delta} + \alpha H)$ is the pullback of a $\QQ$-Cartier $\QQ$-divisor $D$ on $Z$ (\cite{kollar-mori}, Cor. $3.17$), which is necessarily nef.  Furthermore, by projecting to $Z$, we can calculate it:
$$D \sim_{\QQ} K_{Z/\Delta} + \alpha\pi_*H.$$

If $\pi_R$ is a small contraction with flip $\pi^+: V' \rightarrow Z$, then we again find a nef $\QQ$-Cartier $\QQ$-divisor $D$ on $Z$ as above, and look at $D' = (\pi^+)^*D$ on $V'$, which is also nef.  Since $V$ and $V'$ agree in codimension one, we have
$D' \sim_{\QQ} K_{V'/\Delta} + \alpha \tau_*(H)$.
\end{proof}

\subsection{Proof of Theorem \ref{ss-degeneration}}

Let 
$$f: \calY \rightarrow \Delta$$ be the semistable model constructed in Proposition \ref{semistable}.  In particular, we have a birational map 
$$\tau\circ\rho: \calY_{K}= Y' \rightarrow Y \rightarrow X$$ and an ample $\QQ$-divisor
$\calL$ on $\calY$ whose restriction to $Y'$ agrees with the pullback of $L$ on the smooth locus of  
$Y' \rightarrow \PP^1_{K}$, up to $\QQ$-linear equivalence.

We apply semistable MMP with scaling as in the last section.  Since the generic fiber has Kodaira dimension zero, all steps are divisorial contractions or flips along the central fiber.  As an endpoint, we obtain a projective family of surfaces
$$g: \calX \rightarrow \Delta$$
which has singularities of type (**) and for which the canonical divisor
$K_{\calX/\Delta}$ is nef.  Since $X$ is $K$-trivial, we recover the contraction $\tau\circ\rho: Y' \rightarrow X$ as a series of blowdowns when we restrict every step to $\Spec K$.

Furthermore, by Lemma \ref{scaling} we have a $\QQ$-divisor $H$
such that $K_{\calX/\Delta} + H$ is nef.  If we restrict $H$ to $X$, it follows from the statement of the lemma that there exists $0 < \alpha \in \QQ$ such that 
$$H|_{X} \sim_{\QQ} \alpha (\tau\circ\rho)_*(\calL|_{Y'}).$$
By construction, 
the $\QQ$-divisors $\rho_*(\calL|_{Y'})$ and $\tau^*L$ agree (up to $\sim_{\QQ}$) when restricted to $Y_{\mathrm{sm}}$.  Since all fibers of  $\pi:Y\rightarrow \PP^1$ are reduced and irreducible by Lemma \ref{genericpencil}, we have
$$\rho_*(\calL|_{Y'}) \sim_{\QQ} \tau^*L + \epsilon [F]$$
for  $\epsilon \in \QQ$ and $[F]$ is the class of a fiber of $\pi$.
Since $\tau_*[F] = L$, we have that $$H|_{X} \sim_{\QQ} \tau_*(\tau^*L) + \epsilon L =  \gamma L$$ for some $\gamma \in \QQ$.  
Since $\calL|_{Y'}$ is ample, $H|_{X}$ is $\QQ$-linearly equivalent to a nonzero effective $\QQ$-divisor.  So we must have $\gamma > 0$.

The following lemma is proven in Corollary $3.7$ of \cite{corti}.
\begin{lem}  
We have 
$K_{\calX/\Delta} = 0.$  In particular, $K_{\calX/\Delta}$ is Cartier.
\end{lem}
Notice that this is an equality of Weil divisor classes, i.e. not up to $\QQ$-linear equivalence.
\begin{proof}
Since $K_X$ is trivial, by uppersemicontinuity of $H^0(g_*\calO(- K_{\calX/\Delta}))$, $-K_{\calX/\Delta}$ is represented by an effective Weil divisor supported on the central fiber.  If the irreducible components of the central fiber are $S_i$, we have an equality of Weil divisor classes
$$-K_{\calX/\Delta} = \sum a_i [S_i],$$
with integers $a_i \geq 0$.
Since $\sum [S_i] = 0$, we can arrange for $a_i = 0$ for some $i$.  If not all $a_j = 0$, we can choose adjacent components $S_i, S_j$ such that $a_i =0$ and $a_j> 0$.  By choosing a generic curve $C$ in $S_i$ that meets $S_j$ without lying in $S_j$, we see that $-K_{\calX/\Delta}\cdot C > 0$ contradicting the fact that $K_{\calX/\Delta}$ is nef.
Therefore $K_{\calX/\Delta} = 0$.
\end{proof}
%is this upper semicontinuity argument kosher?

\begin{cor}
The $\QQ$-divisor $H$ is relatively big and nef.
\end{cor}
\begin{proof}
Since $H|_{X}$ is a positive rational multiple of $L$ (up to $\sim_{\QQ}$), it is big over $\Spec K$ which implies the statement on the central fiber by upper semi-continuity.  Since $K_{\calX/\Delta}+ H$ is nef, the previous lemma completes the claim.
\end{proof}

\begin{cor}
The singularities of the central fiber $\calX_0$ are rational double points or normal crossings type.
\end{cor}
\begin{proof}
The formal singularity type of the singular points of $\calX$ are classified in Theorem $4.4$ of \cite{kawamata}.
Since $K_{\calX/\Delta}$ is Cartier, these are the only possibilities.
\end{proof}

At this stage, we have shown a partial version in characteristic $p$ of the Kulikov-Pinkham-Persson classification of degenerations of a K3 surface over $\CC$.  Indeed, suppose $\Delta=\Spec R$ where $R$ is a \emph{complete} discrete valuation ring, and we are given a K3 surface $X \rightarrow \Spec K$ with a very ample bundle $L$ of degree $2d$ such that $p > 2d+4$.
After applying Artin's simultaneous resolution \cite{artin-brieskorn} to our construction so far, there exists a finite base change
$\Delta' \rightarrow \Delta$ and a small modification
$$\calX' \rightarrow \calX\times_{\Delta}\Delta'$$
such that $\calX'$ is an algebraic space over $\Delta$ with $\omega_{\calX'/\Delta'}$ trivial.  

The central fiber $\calX'_0$ of 
$$g': \calX' \rightarrow \Delta'$$ has only normal crossings singularities, trivial dualizing sheaf, and (by flatness of $g'$) has cohomology groups $H^j(\calO_{\calX'_{0}})$ equal to those of a K3 surface.   Therefore by the argument in Nakkajima \cite{nakkajima}, it must be a combinatorial K3 surface, i.e., one of three types:
\begin{enumerate}
\item smooth K3 surface (type I)
\item two rational surfaces joined by a chain of ruled surfaces over an elliptic curve (type II)
\item a configuration of rational surfaces whose dual graph is a triangulation of $S^2$. (type III)
\end{enumerate}
Note that this version of the KPP theorem is weaker than the usual formulation, since the total space $\calX'$ may be singular and we require a base change.

To complete the proof of Theorem \ref{ss-degeneration}, we now apply the fact that $X$ is supersingular.

\begin{lem}
The map $g: \calX \rightarrow \Delta$ is smooth.
\end{lem}
\begin{proof}
It suffices to prove this lemma after replacing $\Delta$ with its completion at the closed point, so we can apply the KPP classification above, after a finite base change and small modification.

By Section $2$ of \cite{rsz}, the simultaneous resolution $\calX'$ cannot have normal crossings singularities.  
Indeed, they show that type II and type III combinatorial K3 surfaces have formal Brauer group with height at most $2$; since the height can only jump under specialization, the central fiber must be a smooth K3 surface.  Therefore, $\calX_0$ has at worst double-point singularities.

To rule these out, we argue as follows.
By Theorem $1.1$ of \cite{artin-ss}, Picard ranks of supersingular K3 surfaces are constant under specialization.  Therefore if we again consider the simultaneous resolution, we see that
$$\Pic(\calX')\otimes\QQ = \Pic(\calX'_0)\otimes\QQ.$$

Suppose there exists an exceptional curve $C \subset \calX'_0$.  Then by the above property, there exists a line bundle $A$ on $\calX'$ such that
$$\deg_C A \ne 0.$$

On the other hand, recall that we have arranged for $\Pic(X) = \Pic(X_{\overline{K}})$.  Therefore the line bundle

$$A|_{g'^{-1}(\Delta'\backslash 0)}$$ descends to $X$.
Since $\calX$ is $\QQ$-factorial, after replacing $A$ with a multiple, its descent to $X$ extends to a line bundle $\calA$ on $\calX$.  If $p: \calX' \rightarrow \calX$ is the simultaneous resolution, then
we have by construction that $p^*\calA = A$ over the generic point of $\Delta'$.  Since $\calX' \rightarrow \Delta'$ is smooth, these line bundles must agree everywhere.
Therefore
$$\deg_C A = \deg_C p^*\calA = 0$$
since $p$ contracts $C$ to a point.
This gives a contradiction, so there are no exceptional curves and $g$ is smooth.
\end{proof}

If we return to the statement of Theorem \ref{ss-degeneration}, the last thing to check is that if we extend $L$ (uniquely) to a line bundle $L'$ on $\calX$, then $L'$ is a quasipolarization.  However, we already have a nef $\QQ$-divisor $H$ on $\calX$ that is $\QQ$-linearly equivalent to $\gamma L$ with $\gamma > 0$ when we restrict to $\calX_K = X$.  Therefore, 
$$H \sim_{\QQ} \gamma L'$$ 
on $\calX$
and $L'$ is nef.  Bigness of $L'$ is automatic by upper semicontinuity.  This completes the proof.

\section{Positivity of Hodge bundle}\label{Hodgepositivity}

Given a family of K3 surfaces $\pi: X \rightarrow S$, recall the Hodge bundle on $S$ is the line bundle
$$\lambda_S = \pi_*(\omega_{X/S})\in \mathrm{Pic}(S).$$
If it is clear from context, we will suppress the subscript.
In this section and the next, we prove a positivity result for the Hodge bundle in characteristic $p$.

\begin{thm}\label{amplehodge}
Given $p \geq 5$ and  $d$ such that $p \nodivide 2d$.
Given any map
$g: C \rightarrow \M_{2d,\FF_p}$ from a smooth proper curve $C$ such that
\begin{enumerate}
\item $C$ is not contracted in the map to the coarse moduli space and
\item $C$ meets the polarized locus $\M^{\circ}_{2d,\FF_p}$ nontrivially,
\end{enumerate} 
the pullback of the Hodge bundle on $\M_{2d,\FF_p}$ to $C$ has positive degree.
Furthermore, the restriction of the Hodge bundle to $\M^{\circ}_{2d,\FF_p}$ is ample.
\end{thm}

\begin{remark} The condition on meeting the polarized locus should not be essential, but is a technical condition that makes the arguments in the next section simpler.  If one adapts the techniques there to apply for algebraic spaces instead of schemes, the theorem can be strengthened to show ampleness on any separated substack.
\end{remark}

In characteristic zero, these results are well-known via Hodge theory.  Unfortunately, we do not know a similar argument in finite characteristic, so the proof we give here is extremely indirect.  A more geometric
approach to this question would be very interesting.  

\begin{remark}  It is possible to shorten the proof given here using the theory of integral canonical models of Shimura varieties in Section \ref{qfks}, along the lines of the unpublished manuscript of Vasiu \cite{vasiu}.  Since we prefer not to use that technology (and find the arguments there difficult to follow), we provide a longer argument here.
\end{remark}

Let us state a couple of immediate geometric corollaries of Theorem \ref{amplehodge}.
\begin{cor}
The moduli space $\M^{\circ}_{2d,\FF_p}$ is a quasiprojective Deligne--Mumford stack, when $p \geq 5$ and $p \nodivide 2d$.
\end{cor}
The following corollary follows from Theorem $15.3$ in \cite{vandergeer-katsura}:
\begin{cor}
Given a family $f: \calX \rightarrow B$ of polarized K3 surfaces of degree $2d$ with $p \geq 5$ and $p \nodivide 2d$, such that $B$ is proper and $f$ is not isotrivial, either
\begin{enumerate}
\item the height of the fibers $\calX_t$ are not constant, or
\item all fibers $\calX_t$ are supersingular.
\end{enumerate}
\end{cor}

The strategy of the proof of Theorem \ref{amplehodge} is as follows.  
Given a family of abelian varieties $$\pi: A \rightarrow S,$$ 
the (determinant) abelian Hodge bundle on $S$ is the line bundle
$\lambda_{\A,S}= \pi_*(\omega_{A/S}).$
Again, we will suppress the subscript for the base $S$ if clear from context.
In arbitrary characteristic, we know from \cite{faltings-chai} that the abelian Hodge bundle is ample on the moduli space 
of abelian varieties.

In order to apply this result to our situation, we use the Kuga--Satake construction to define a map between the moduli spaces of K3 surfaces and abelian varieties.  While a priori this is a transcendental construction, it follows from work of Deligne, Andr\'e, Rizov, and Vasiu, that it behaves well in families in mixed characteristic (under some hypotheses).
In this section, we study this map and show that it preserves Hodge bundles, up to a multiple. This will prove Theorem \ref{amplehodge}, assuming the Kuga--Satake map is quasifinite.  We will prove this quasifiniteness in the next section.

\subsection{Kuga--Satake over $\CC$}\label{kugasatakecc}

We first briefly recall the Kuga--Satake construction for a K3 surface $X$ over $\CC$, equipped with a quasipolarization $L$ of degree $2d$, following the description in \cite{deligne-weil} and \cite{huybrechts}.

For ease of notation, we identify the integral primitive cohomology of $(X,L)$ with the lattice $L_{2d}$ from Section \ref{cliffordsection}.  It carries a polarized weight $2$ Hodge structure, determined by the choice of line
$$F^2 = \CC\omega \subset L_{2d}\otimes \CC$$
such that $\langle \omega,\omega\rangle = 0$.

To construct the associated Kuga--Satake abelian variety as a complex torus, we take the free $\ZZ$-module $\Cl_+(L_{2d})$ and put a complex structure on 
the real vector space $$\Cl_+(L_{2d})\otimes \RR = \Cl_+(L_{2d}\otimes \RR).$$  
For this, we choose the generator $\omega$ of $F^2$ such that $\langle \omega, \overline{\omega}\rangle = -2$, where  $\overline{\omega}$ denotes complex conjugation.  If we write $\omega = x+iy$, where $x,y \in L_{2d}\otimes \RR$, then the complex structure $J$ on 
$\Cl_+(L_{2d}\otimes \RR)$ is given by the operator of left multiplication by $x\cdot y$, with respect to the Clifford algebra structure.

To show that this complex torus is in fact an abelian variety, we define a polarization on the weight $1$ Hodge structure as follows.  Given the hyperbolic lattice $U$ with standard basis $e$, $f$, the element $v = e-f$ has self-intersection $\langle v,v\rangle = -2$.  In our explicit presentation of $L_{2d}$ from Section \ref{cliffordsection}, define $v_1,v_2 \in L_{2d}$ by taking the corresponding element in the second and third copy of $U$ in equation \eqref{k3lattice}.  
We define a skew-symmetric pairing on $\Cl_+(L_{2d})$ by the formula
$$\langle x, y \rangle  = \mathrm{Tr}(x^* y v_1 v_2) $$
where $\mathrm{Tr}$ is the trace of the operator of left multiplication on $\Cl(L_{2d})$.
It is proven in \cite{huybrechts} that (depending on the orientation of $\omega$), either $\langle, \rangle$ or $-\langle,\rangle$ defines a polarization.

In either case, we have an abelian variety of dimension $2^{19}$.  By calculation, one can show that the explicit polarization we have given has degree $(d')^2$ with 
$d' =  2^{3\cdot 2^{18}}\cdot d^{2^{19}}.$  All that matters for us is that it is only divisible by primes dividing $2d$.

Finally, if we define the finitely generated $\ZZ$-algebra
$$C^{+} = \Cl_{+}(L_{2d})^{\mathrm{op}},$$
it is clear that we have an action of $C^+$ on 
$\Cl_{+}(L_{2d})$ by right multiplication that preserves the complex structure just defined.  Therefore, the Kuga--Satake abelian varieties 
carry an action of $C^+$.

\subsection{Field of definition}

From now on, we restrict to level $n=4$.  In the next section, we will explain results of Andr\'e and Rizov regarding 
Kuga--Satake morphisms
in families over mixed-characteristic base.  See also \cite{vasiu} for related constructions.
To apply their results, we need to work over an appropriate number field $E$ satisfying certain properties. We study this
field extension in this section.

Recall from Section \ref{cliffordsection} the norm map $\mathrm{Nm}: \mathrm{CSpin}(V_{2d})\rightarrow \mathbb{G}_m$ and consider the image of
the compact open subgroup $\KK_n \subset \mathrm{CSpin}(V_{2d})(\mathbb{A}_f)$:
$$\mathrm{Nm}(\KK_{n}) \subset (\ZZhat)^*.$$

We claim that it contains a congruence subgroup $$U_m = \{a \in \ZZhat^*| a \equiv 1\bmod m\}$$ for some $m$ only divisible by prime factors of $2nd=8d$.  
Indeed, since $\KK_n$ is compact open, so is its image under $\mathrm{Nm}$.   
For each prime $l$ not dividing $2d$, $$\Cl_+(L_{2d}\otimes \ZZ_l)^*\cap \mathrm{CSpin}(V_{2d}\otimes \QQ_l)$$ is smooth over $\ZZ_l$ and so surjects onto $\ZZ_l^*$ (using Hensel's lemma and the surjection on $\FF_l$-points).  
%Serre's conjecture gives surjection on finite field points, or Lang's theorem whatever

We require a finite extension $E$ of $\QQ(\zeta_m)$ such that every connected component of
$\M_{2d, n, E}$ is geometrically connected.

\begin{lem}
There exists a finite extension $E/\QQ$ as above that is unramified away from primes dividing $2d$.
\end{lem}
\begin{proof}

Since $\M_{2d,\QQ}$ is geometrically connected, to show that every connected component of 
the finite \'etale cover $\pi: \M_{2d,n,E} \rightarrow \M_{2d,E}$ is geometrically connected (for some choice of $E$),
it suffices to find an $E$-point, $$\mathrm{Spec} E \rightarrow \M_{2d,E},$$ such that the pullback of $\pi$ is a trivial Galois cover.  
Equivalently, we need to find a quasipolarized K3 surface over $E$ such that all spin level $n$ structures on $X$ are defined over $E$.  
In order to do this, choose $N$ divisible only by the primes dividing $2d$ such that $$\KK^{ad}_{n} \supset \{ g \in \mathrm{Aut}(L \otimes \ZZhat)| g \equiv 1 \bmod N\}.$$
We want to find a number field $E$ and a quasipolarized K3 surface over $E$ such that $\mathrm{Gal}(\overline{\QQ}/E)$ acts trivially
on $H^{2}_{\et}(X_{\overline{\QQ}},\ZZ/N\ZZ(1))$.  

Again, it suffices to find a number field $E'$, unramified over $\QQ$ away from $2d$,  and a family of K3 surfaces  over $\mathcal{O}_{E'}[1/2d]$ with a quasipolarization of degree $2d$.  If we have this, the extension of $E'$ defined by $H^{2}_{\et}(X_{\overline{\QQ}},\ZZ/N\ZZ(1))$ will necessarily be unramified away from $2d$.  For this, we can take the Fermat quartic over $E' = \QQ(\zeta_8)$ (over which all line bundles are defined).  It has the structure of an elliptic fibration with section by Lemma $12.2.2$ of \cite{schuett-shioda}, so it has quasipolarizations of every degree, all defined over $E'$, and good reduction away from $2$.
\end{proof}

\subsection{Kuga--Satake construction in families}

In what follows, let
$$\A_{g,d',n}$$
denote the moduli stack of abelian varieties of dimension $g$ with a polarization of degree $d'^2$ and a level $n$ structure, where $d'$ is the polarization degree defined in Section \ref{kugasatakecc}.
It is a smooth quasiprojective scheme over $\ZZ[1/d'n]$ (see (\cite{mumford}, 7.9), and (\cite{oort}, 2.4.1)).

We can now state slightly modified versions of Theorem $8.4.3$ from \cite{andre} and Lemma $5.5.5$ and Theorem $6.2.1$ from \cite{rizov-total}, which allow us to extend the Kuga--Satake construction to families over mixed-characteristic base.

\begin{prop}\label{kugasatakerizov}
\begin{enumerate}
\item
There exists a morphism
$$\kappa_E: \M_{2d, n, E} \rightarrow \A_{g,d',n,E}$$
which on $\CC$-valued points sends a polarized K3 surface to its associated Kuga--Satake variety.

\item
If $\pi: \calA_\kappa \rightarrow \M_{2d, n, E}$ denotes the abelian scheme given by pulling back the universal family over
$\A_{g,d',n}$, then there exists an isomorphism
$$C^+ \rightarrow \mathrm{End}_{\M_{2d,n,E}}(\calA_\kappa)$$
of algebras.
Here $C^+$ is the algebra defined in the last section.

\item
There exists a unique isomorphism of \'etale sheaves of algebras
$$\Cl_+(\Pet(\ZZ_l(1))) \simrightarrow \mathrm{End}_{C^{+}}(\Ret^1\pi_*(\ZZ_l))$$
where the left-hand side denotes the Clifford construction for the universal family of K3 surfaces  $f: X \rightarrow \M_{2d,n,E}$
and the right-hand side denotes the sheaf of endomorphisms of $\Ret^1\pi_*(\ZZ_l)$ that commute with the action of $C^+$.

\item
The morphism $\kappa_E$ extends to a morphism
$$\kappa: \M_{2d, n, \calO_E[1/2d]} \rightarrow \A_{g,d',n,\calO_E[1/2d]}.$$
\end{enumerate}
\end{prop}

The morphism $\kappa$ in the proposition is not canonical; there exist many choices satisfying the above conditions.
\begin{proof}

If we restrict to the moduli space of polarized K3 surfaces, this proposition is precisely what is shown in \cite{andre, rizov-total}.
The arguments given there easily extend to the quasipolarized case, as we now explain.

For part (i), the proof in the polarized case proceeds as follows.
Recall the algebraic groups $G = \mathrm{SO}(V_{2d})$ and $G_1 = \mathrm{CSpin}(V_{2d})$ over $\QQ$; if we take the symplectic rational vector space
$\calW = \Cl_+(V_{2d})$ with skew-symmetric pairing $\langle,\rangle$ from Section \ref{kugasatakecc}, we also have the algebraic group $\mathrm{CSp}(\calW)$ of symplectic transformations.  The adjoint
and spin representations define maps
$$\mathrm{ad}: G_1 \rightarrow G;\ \ \mathrm{sp}: G_1 \rightarrow \mathrm{CSp}(\calW).$$

Recall again the (disconnected) Hermitian symmetric domain 
$$\Omega^{\pm} = \{\omega \in V_{2d}\otimes \CC| \langle\omega, \omega\rangle = 0, \langle\omega,\overline{\omega}\rangle > 0\}.$$
There exists a unique Hermitian symmetric domain $\Omega_1^{\pm}$ for $G_1$ such that the adjoint map
extends to a map
$$(G_1, \Omega_1^{\pm}) \rightarrow (G, \Omega^{\pm})$$
of Shimura data, and so that 
$\Omega_1^{\pm} \rightarrow \Omega^{\pm} $ is an analytic isomorphism.
Finally, if $\mathcal{H}^{\pm}$ is the union of the Siegel upper and lower half-planes for the symplectic space $W$, we have a map of Shimura data
$$(G_1, \Omega_1^{\pm}) \rightarrow (\mathrm{CSp}(\calW), \mathcal{H}^{\pm}).$$
The associated Shimura varieties all have canonical models over $\QQ$, so we have maps of Shimura varieties
\begin{align}
\pi_{\mathrm{ad}}: \mathrm{Sh}_{\KK_{n}^{\mathrm{sp}}}(G_1, \Omega_1^{\pm}) \rightarrow \mathrm{Sh}_{\KK_{n}^{\mathrm{ad}}}(G, \Omega^{\pm}) \\
\pi_{\mathrm{sp}}:  \mathrm{Sh}_{\KK_{n}^{\mathrm{sp}}}(G_1, \Omega_1^{\pm})\rightarrow \mathrm{Sh}_{\Lambda_{n}}(\mathrm{CSp}(\calW), \mathcal{H}^{\pm})
\end{align}
where
$\Lambda_{n}$ is the congruence level-$n$ subgroup of $\mathrm{CSp}(\calW)(\mathbb{A}_f)$ associated to the integral lattice $\Cl_+(L_{2d})$ of $\calW$.
Finally, we can identify $\mathrm{Sh}_{\Lambda_{n}}(\mathrm{CSp}(\calW), \mathcal{H}^{\pm})$ with a connected component of $\A_{g,d',n,\QQ}$.

In \cite{rizov-total}, it is shown that, after base change to $E$ as in the last section, all connected components of the Shimura varieties for $G_1$ and $G$ are geometrically connected, and $\pi_{\mathrm{ad}}$ restricts to an isomorphism on each such component.  Therefore, after base change to $E$, there exists (after a choice) a section
$$\sigma: \mathrm{Sh}_{\KK_{n}^{\mathrm{ad}}}(G, \Omega^{\pm})_{E}\rightarrow \mathrm{Sh}_{\KK_{n}^{\mathrm{sp}}}(G_1, \Omega_1^{\pm})_{E}.$$

With all this structure in place, 
the proof of part (i) is to first construct a period morphism
$$j_{2d,n}: \M^{\circ}_{2d,n^{\mathrm{sp}},\QQ} \rightarrow  \mathrm{Sh}_{\KK_{n}^{\mathrm{ad}}}(G, \Omega^{\pm})$$
and then define 
$$\kappa = (\pi_{\mathrm{sp}}\otimes E)\circ \sigma \circ (j_{2d,n}\otimes E).$$

To extend this construction to the quasipolarized case, it suffices to extend the definition of $j_{2d,n}$.
After tensoring with $\CC$, the construction in \cite{rizov-total} of the map applies unchanged since a quasipolarized family still gives rise to a polarized variation of weight $2$ Hodge structures of K3 type.  Since there exists a Zariski dense open subset (namely the polarized locus) for which $j_{2d,n}$ descends to $\QQ$, the entire map descends as well.

For parts (ii) and (iii), once we know these conditions on a Zariski dense open subset of $\M_{2d,n}$, we know them on the total space.  

For part (iv), the argument in \cite{rizov-total} extends easily. 
Since $\M^{\circ}_{2d,n,\FF_p}$ is a Zariski dense open subset of $\M_{2d,n,\FF_p}$, we already know the abelian scheme (and auxiliary structure) extends over codimension one points.  To extend over all higher codimension points, the argument in Theorem $6.2.1$ applies without change.
\end{proof}

\subsection{Comparison of Hodge bundles}

In this section, we compare the Hodge bundle $\lambda$ on the moduli space of K3 surfaces with the pullback of the determinant abelian Hodge bundle $\lambda_{\A}$ on the moduli of abelian varieties, via the Kuga--Satake map $\kappa_E$.

\begin{prop}\label{hodgecompare}
There exists a positive integer $a$ such that
$$\kappa_{E}^*(\lambda_{\A}^{\otimes a}) = \lambda^{\otimes (2^{20}\cdot a)} \in \mathrm{Pic}(\M_{2d,n,E}).$$
\end{prop}
\begin{proof}

Given a line bundle $L$ on a finite type scheme $X$ over $E$, if it is trivial after pullback to $X\otimes \CC$, then it is torsion on $X$.  Indeed, if it is trivial on $X \otimes \CC$, a standard spreading-out argument shows that it is trivial on $X \otimes \overline{\QQ}$ also and therefore over some finite extension of $E$ of degree $r$.  By taking the norm of a trivialization defined over this finite extension, we obtain a trivialization of $L^{\otimes r}$ defined over $E$.

It therefore suffices to prove the claim over $\M_{2d,n,\CC}$.
For this, we rephrase everything in terms of automorphic bundles on Shimura varieties.  We refer the reader to section III of \cite{milne-autbundles} for a detailed overview of this subject.
Given Shimura data $(H,X)$, a point $x\in X$ determines a parabolic subgroup $P_x \subset H(\CC)$.  One can associate, to a finite-dimensional representation $\phi$ of $P_x$, an automorphic bundle $\calV_{\phi}$ on the Shimura variety $\mathrm{Sh}_{\KK}(H,X)_{\CC}$.  If the representation $\phi$ extends to a representation of $H(\CC)$, then the algebraic bundle $\calV_\phi$ carries an integrable connection and, in particular, can be associated to a local system. The local system is determined by the restriction of $\phi$ to $\KK \cap H(\QQ) \subset H(\CC)$ (see \cite{milne-autbundles}, section III.2).

For the orthogonal group $G$, if we fix a point on $\Omega^{\pm}$ corresponding to a line $\CC\omega \subset V_{2d}\otimes \CC$, the parabolic subgroup
is $$P_{\mathrm{ad}} = \{ g \in G(\CC)| g(\CC\omega) = \CC\omega\}.$$  
Let $P_{\mathrm{sp}}\subset G_1(\CC)$ denote the preimage of $P_{\mathrm{ad}}$ with respect 
to the adjoint map.  Finally, let $P_{\mathrm{ab}}$ denote the subgroup of $\mathrm{CSp}(\calW)(\CC)$ that preserves the weight $1$ Hodge filtration on $\calW$ induced by $\omega$ via the Kuga--Satake construction.

Let $\rho$ denote the one-dimensional representation of $P_{\mathrm{ad}}$ given by its action on the one-dimensional subspace $\CC\omega$.
It follows from the definitions that
$$\lambda = j_{2d,n}^*(\calV_{\rho}),$$
where $\calV_\rho$ denotes the automorphic bundle associated to $\rho$ on $\mathrm{Sh}_{\KK_n}(G,\Omega^{\pm})$.

Similarly, if $\phi: P_{\mathrm{ab}} \rightarrow \mathrm{End}(\mathrm{Fil}^1)$ denotes the action on the first filtered piece of the weight one Hodge structure on $W$, then
$$\lambda_{\A} = \mathrm{det}(\calV_{\phi}).$$
By restriction, we have a representation of $P_{\mathrm{sp}}$ on $\mathrm{Fil}^1$; the associated automorphic bundle is the pullback of $\calV_\phi$ by 
the map $\pi_{\mathrm{sp}}$.

We first show the following lemma.

\begin{lem}\label{psp-reps}
There is a short exact sequence of representations of $P_{\mathrm{sp}}$:
$$0 \rightarrow \mathrm{Fil}^1 \rightarrow \Cl_+(V_{2d}\otimes\CC) \rightarrow \rho^{-1} \otimes \mathrm{Fil}^1 \rightarrow 0.$$
\end{lem}

\begin{proof}
It follows from definitions that
$$\mathrm{Fil}^1 = \{a \in \Cl_+(V_{2d}\otimes\CC)| \omega a = 0\} = \{a \in \Cl_+(V_{2d}\otimes\CC)| a = \omega \cdot b\},$$
where the action of $P_{\mathrm{sp}}$ is by left multiplication.
Fix an element $ v\in V_{2d}$ such that $\langle v,v\rangle \ne 0$ and consider the linear map
$$\Cl_+(V_{2d}\otimes\CC) \rightarrow \Cl_+(V_{2d}\otimes\CC)$$
given by
$$x \mapsto \omega\cdot x\cdot v.$$
Since right multiplication commutes with the action of $P_{\mathrm{sp}}$, this map intertwines the parabolic action after we twist by $\rho$.
Furthermore, since right multiplication by $v$ is invertible, both the kernel and image of this map are $\mathrm{Fil}^1$.
\end{proof}

Let $$\tau = \det \Cl_+(V_{2d}\otimes\CC)$$ denote the one-dimensional representation of $P_{\mathrm{sp}}$.  Since this representation extends to
a representation of $G_1$, the line bundle $\calV_\tau$ is associated to a local system.  Furthermore, by \cite{andre}, section $4.2$, the compact open subgroup
$\KK_{n}^{\mathrm{sp}}$ is contained in $\mathrm{Spin}(V_{2d})(\mathbb{A}_f)$.  Since $\mathrm{Spin}$ is semisimple, the restriction of $\tau$ will be trivial, therefore the restriction of $\tau$ to the congruence subgroup $\KK_{n}^{\mathrm{sp}} \cap G_1(\QQ)$
must be trivial, and $\calV_\tau$ is the trivial bundle.

If we take the determinant of the short exact sequence of Lemma \ref{psp-reps}, and apply this triviality, we see that
$$(\pi_{\mathrm{sp}}^*\det\calV_\phi)^{\otimes 2} = (\calV_\rho)^{\otimes 2^{21}}.$$
as line bundles on $\mathrm{Sh}_{\KK_{n}^{\mathrm{sp}}}(G_1,\Omega_1^{\pm})_{\CC}$.

If we pull everything back to $\M_{2d,n,\CC}$ using the period map $j_{2d,n}$ and $\sigma$, the statement of the proposition follows immediately.

\end{proof}
 \subsection{Proof of Theorem \ref{amplehodge}}\label{proofample}

Fix a prime $p \geq 5$ such that $p$ does not divide $2d$.  Let $\kk = \overline{\FF}_p$ and $W = W(\kk)$ be the the ring of Witt vectors with fraction field
$K = W[1/p]$.
Since the number field $E$ is unramified at primes dividing $p$, we have a map 
$$\calO_{E}[1/2d] \rightarrow W(\kk)$$
and we can study the base-change of the Kuga--Satake morphism
$$\kappa_W: \M_{2d,n,W} \rightarrow \A_{g,d',n,W}.$$

In the next section we show the following

\begin{prop}\label{quasifinite}
The Kuga--Satake map $\kappa_{\kk}$ in characteristic $p$ is quasifinite when restricted
to the polarized locus:
$$\kappa_{\kk}:\M^{\circ}_{2d,n,\kk} \rightarrow \A_{g,d',n,\kk}.$$
\end{prop}

If we restrict $\kappa_{\kk}$ to the ordinary locus (i.e. K3 surfaces whose formal Brauer group have height $1$), then quasifiniteness is proven in \cite{rizov-total} using the theory of canonical lifts.

Let us take this proposition for granted and finish the proof of Theorem \ref{amplehodge}.
First, observe that it suffices to prove both parts of Theorem \ref{amplehodge} after base change to $\kk$ and also after passing to the finite \'etale cover $\M_{2d,n,\kk}$.

We first show two elementary lemmas about algebraic spaces.
\begin{lem}
Given a smooth algebraic space $X$ of finite type, a line bundle $L$ on $X$, and a reduced and irreducible divisor $Z \subset X$ such that 
$L$ is trivial on $X^{\circ} = X \backslash Z$,
there exists an integer $a$ such that
$L = \calO(aZ)$.
\end{lem}
\begin{proof}
Pick a trivialization
$$\psi: \calO|_{X^{\circ}} \rightarrow L|_{X^{\circ}}.$$
Choose a presentation
$R \rightarrow U \times U$ 
of $X$ as a quotient by an \'etale equivalence relation.  Let $Z_U$ and $Z_R$ denote the preimages of $Z$ in $U$ and $R$ respectively, and let $U^{\circ}$ and $R^\circ$ denote the open Zariski-dense complements.  
If we pullback $\psi$ to a trivialization $\psi^{\circ}_{U}$, over $U^{\circ}$, we can evaluate its order of vanishing along each irreducible component of $Z_U$.  
Since $Z$ is irreducible, given any two irreducible components of $Z_U$, there exists an irreducible component of $Z_R$ dominating them both with respect to the two projections from $R$ to $U$.  Therefore, since $\psi$ descends to $X^{\circ}$, by pulling back to $R$, we see that its order of vanishing is the same on each irreducible component of $Z_U$.  After twisting appropriately, $\psi^{\circ}_{U}$ extends to an isomorphism
$$\psi_{U}: \calO(a Z_U) \rightarrow L_U$$
which descends to an isomorphism on $X$ since its restriction to $U^{\circ}$ descends.
\end{proof}

\begin{lem}\label{specialize-line}
Let $X$ be an algebraic space, equipped with a smooth, finite-type map 
$$f: X \rightarrow \Spec W,$$
and a line bundle $L$ on $X$.  If the restriction $L_K$ to the generic fiber of $f$ is trivial, then
so is the restriction $L_\kk$ to the special fiber.
\end{lem}

\begin{proof}
If we let $Z_i$ denote the connected components of $X_{\kk}$, then by smoothness of $f$, each $Z_i$ is irreducible and reduced and is a Cartier divisor on $X$.  Since they are disjoint, we have
$$\calO_{Z_i}(Z_j) = \calO_{Z_i}$$
for $j \ne i$.
Also, since the central fiber $X_{\kk} = \sum_{j} Z_j$ is principal, 
we have 
$$ \calO_{Z_i}(Z_i) = \calO_{Z_i}(\sum_{j} Z_j)  = \calO_{Z_i}.$$

Since $L_K$ is trivial, by the previous lemma applied iteratively,
there exist integers $a_i$ such that 
$$L = \calO_X(\sum a_i Z_i).$$
Therefore, its restriction to each component of $X_{\kk}$ is trivial.
\end{proof}

Since $\M_{2d,n,W}$ is smooth over $\Spec W$, Proposition \ref{hodgecompare} and Lemma \ref{specialize-line} imply that
$$\kappa_{\kk}^*(\lambda_{\A}^{\otimes a}) = \lambda^{\otimes (2^{20}\cdot a)} \in \mathrm{Pic}(\M_{2d,n,\kk}).$$

By Theorem $V.2.3$ of \cite{faltings-chai}, the determinant abelian Hodge bundle $\lambda_{\A}$ is ample on $\A_{g,d',n,\kk}$.  Therefore, Proposition \ref{quasifinite} immediately implies that $\lambda$ is positive on any curve that intersects $\M^{\circ}_{2d,n,\kk}$.  Furthermore, if we restict to the polarized locus $\M_{2d,n,\kk}^{\circ}$, the map $\kappa_{\kk}$ is separated, so by Zariski's Main Theorem, it is a composition of an open immersion and a finite morphism to $\A_{g,d',n,\kk}$.  Therefore the pullback of $\lambda_{\A}$, and so also $\lambda$, are ample.

\section{Quasifiniteness of Kuga--Satake}\label{qfks}

In this section, we prove Proposition \ref{quasifinite}, stating that the Kuga--Satake morphism 
$$\kappa_{\kk}: \M^{\circ}_{2d,n,\kk} \rightarrow \A_{g,d',n,\kk}$$
 is quasifinite on the polarized locus.

Suppose otherwise.  Then there exists an unramified map
$j_{\kk}: B_{\kk} \rightarrow \M^{\circ}_{2d,n,\kk}$ from a
smooth affine curve $B_{\kk}$ over $\kk$ such that $\kappa_{\kk}\circ j_{\kk}$ contracts $B_{\kk}$.
Since $\M^{\circ}_{2d,n,W}$ is smooth over $\Spec W$, after possibly shrinking $B_{\kk}$, there exists a lift
$$B = \Spec A \rightarrow \Spec W$$
where $B$ is smooth affine scheme of relative dimension $1$ and a map
$$j: B \rightarrow \M^{\circ}_{2d,n,W}$$
that specializes to $j_{\kk}$.
%use Elkik

After passing to an \'etale neighborhood, we can assume $j$ arises from a polarized
family of K3 surfaces 
$$f: X \rightarrow B$$
where $X$ is a scheme, with polarization induced by a line bundle $L$ on $X$.  We can assume by further shrinking
that there exists an \'etale map $W[t,t^{-1}] \rightarrow A$.
%and a lift $\phi: A \rightarrow A$ that lifts the Frobenius action on $A \otimes_{W} \kk$.

The composition $\kappa_W\circ j$ and Proposition \ref{kugasatakerizov} gives us a relative abelian scheme
$$\pi: \calA \rightarrow B$$
equipped with a fiber-wise action
$$C^+ \rightarrow \End_{B} \calA$$
such that, 
over $\Spec A[1/p]$, we have an isomorphism of \'etale sheaves
\begin{equation}\label{etaleclifford}
\Cl_{+}(\Pet\ZZ_p(1)) \simrightarrow \End_{C^{+}}(\Ret^1\pi_*(\ZZ_p)).
\end{equation}

Given the family (over $\Spec \kk$) $$f_\kk: X_\kk \rightarrow B_\kk,$$ 
recall we have the Gauss--Manin connection on relative de Rham cohomology:
$$\Del: R^2f_{\kk,*}(\Omega^*_{X_{\kk}/B_{\kk}}) \rightarrow R^2f_{\kk,*}(\Omega^*_{X_{\kk}/B_{\kk}}) \otimes \Omega^1_{B_{\kk}}.$$
If we pass to the associated graded with respect to the Hodge filtration, we have the $\calO_{B_{\kk}}$-linear map
$$\mathsf{gr}^{2}\Del: f_*(\Omega^2_{X_{\kk}/B_{\kk}}) \rightarrow R^1f_*(\Omega^1_{X_{\kk}/B_{\kk}})\otimes \Omega^1_{B_{\kk}},$$
known as the Kodaira-Spencer map.  After taking duals, $\mathsf{gr}^{2}\Del$ is identified with the differential of $j_{\kk}$ which by assumption is nonvanishing.
%\Dcom{reference for this identification?}

Therefore, quasifiniteness is reduced to the following proposition.

\begin{prop}
If the family $\calA_\kk \rightarrow B_\kk$ is trivial, then the Kodaira-Spencer map
$\mathsf{gr}^{2}\Del$ for $f_{\kk}:X_{\kk}\rightarrow B_{\kk}$ vanishes identically.
\end{prop}

To prove this proposition, we use comparison theorems from p-adic Hodge theory to translate 
\eqref{etaleclifford} into an isomorphism of associated filtered Frobenius crystals.  By reducing mod $p$, this isomorphism
allows us to study the de Rham cohomology over $\kk$.

\subsection{Setup}\label{section-setup}

We first replace the ring $A$ with its $p$-adic completion, i.e., we replace $B$ with the $p$-adically complete scheme
$$\Bhat = \Spec \Ahat$$
and with the families obtained via base change
$$f: \Xhat \rightarrow \Bhat,\qquad \pi: \calAhat \rightarrow \Bhat$$
over $\Spec W$.
We can assume that there is a lift of the Frobenius morphism $\phi$ on $A\otimes \kk$ to $\Ahat$. 

\begin{remark}\label{setup}

In this section and the next, we will recall some results and constructions from $p$-adic Hodge theory, specialized to our setting.  In what follows,
let $R$ be an integral domain that is the $p$-adic completion of a smooth $W$-algebra, equipped with a lift of Frobenius on $R\otimes \kk$, and let $\Shat = \Spec R$, $R[1/p] = R\otimes_{W} K$ and $\Shat[1/p] = \Spec R[1/p]$.

We will work with lisse \'etale $\ZZ_p$-sheaves and $\QQ_p$-sheaves on $\Shat[1/p]$.  If we pick a geometric point $\sbar: \Spec \Omega \rightarrow \Shat[1/p]$, we can think of such a sheaf as a finite free $\ZZ_p$-module (respectively, $\QQ_p$-module) with a continuous action of the profinite group $\pi_1(\Shat[1/p], \sbar)$.

In our situation, the natural base change functor from \'etale sheaves on $B[1/p]$ to \'etale sheaves on $\Bhat[1/p]$ gives rise to an analog of equation \eqref{etaleclifford} over $\Bhat[1/p]$.
\end{remark}
%\begin{def} An \'etale $\ZZ_p$-local system $\LL$ over $\Bhat[1/p]$ is an inverse system  \{\cdots \rightarrow L_1 \rightarrow L_0\} of torsion $\ZZ_p$-modules on the small \'etale site of $\Bhat[1/p]$ such that each $L_n$ is a sheaf $\ZZ/p^n\ZZ$-modules that is lisse (i.e. represented by a finte \'etale scheme over $\Bhat[1/p]$) and such that the transition maps $L_{n+1} \rightarrow L_{n}$ identifies $L_n$ with the cokernel of multiplication by $p^n$.  Similarly, 
%of lisse \'etale sheaf of $\ZZ_p$-modules on $\Bhat[1/p]$ (respectively $\QQ_p-modules).  
%\end{def}

\begin{defn} A filtered Frobenius crystal over $\widehat{S}$ is given by the data $(\calE, \Fil^i, \Del, \Phi)$ where
\begin{itemize}
\item $\calE$ is a locally free sheaf on $\Shat$ of finite rank,
\item $\Fil^i$, for $i \in \ZZ$ is a decreasing filtration of $\calE$ by locally direct summands,
\item $\Del$ is an integrable, topologically quasi-nilpotent connection $\Del: \calE \rightarrow \calE \otimes \Omega^{1}_{\Shat/W}$
satisfying Griffiths transversality, and
\item a horizontal isomorphism $\phi: (\calE \otimes K)\otimes_{\phi} \Shat[1/p] \rightarrow \calE\otimes K.$
\end{itemize}
\end{defn}
In the above, $\Omega^1_{\Shat/W}$ is the module of separated differentials of $R$ over $W$.  Notice that we have not required $\phi$ to be defined integrally (as opposed to an F-crystal). The Frobenius structure will not play a major role in our discussion.
%\Dcom{Does Frobenius need to be horizontal?}

Similarly, we can define a filtered F-isocrystal by working everywhere with $\Shat[1/p]$.
The category of filtered Frobenius crystals is closed with respect to duals and tensor products.

\begin{example}
For the base $\Shat = \Spec W$, we define the filtered Frobenius crystal $W\{-1\}$ to be the free $W$-module of rank $1$ with generator $e \in \Fil^1 \backslash \Fil^2$ and  $\phi(e) = p\cdot e$.  We can define $W\{k\}$ for $k \in \ZZ$ in the natural way, as well as Tate twists $\calE\{k\} = \calE \otimes_W W\{k\}$ of a filtered Frobenius crystal over $\Shat$.
\end{example}

\begin{example}
Given our family $f: \Xhat \rightarrow \Bhat$, relative de Rham cohomology defines a vector bundle on $\Bhat$
$$R^2f_*(\Omega^*_{\Xhat/\Bhat})$$
which can be equipped with the structure of a filtered Frobenius crystal.  The filtration and connection come from the Hodge filtration and Gauss--Manin connection, respectively, while the Frobenius structure comes from its identification with crystalline cohomology, once we have the lift of Frobenius to $\Bhat$, by Remark $2.23$ and Theorem $3.8$ in \cite{ogus2}.
%\Dcom{reference?}
\end{example}

\begin{remark}
The negative Poincar\'e pairing 
$$\psi_{\DR}:  R^2f_*(\Omega^*_{\Xhat/\Bhat})\{1\} \otimes R^2f_*(\Omega^*_{\Xhat/\Bhat})\{1\} \rightarrow \calO_{\Bhat}$$
is compatible with the Gauss--Manin connection, filtration, and Frobenius action (after twisting).
%\Dcom{more references needed!}
Therefore, these structures can be restricted to the primitive cohomology
$P^2_{\DR}(f)\{1\}$ defined in Example \ref{clifford-dR}.
The Clifford algebra
$$\Cl_+(P^2_{\DR}(f)\{1\})$$
inherits the structure of a filtered Frobenius crystal on $\Bhat$ (the filtration was discussed in Example \ref{clifford-dR}; the other structures descend from the even tensor algebra to the Clifford algebra using the compatibility with $\psi_{\DR}$.)
\end{remark}

%The polarization $L$ defines an orthogonal splitting
%\begin{equation}\label{primitivesplitting}
%R^2f_*(\Omega_{\DR})\{1\} = \calO_{\Bhat} \oplus P^2_{\DR}(f)\{1\}
%\end{equation}
%where $P^2_{\DR}(f)$ is primitive de Rham cohomology which inherits the nondegenerate pairing $\psi_{\DR}$.
 
%As in \'etale cohomology, we can define the Clifford algebra
%$\Cl(P^2_{\DR}(f)\{1\})$, which is a filtered Frobenius crystal on $\Bhat$.  As vector bundles, it is defined in the standard way using the pairing $\psi_{\DR}$.  As $\Rhat$-modules, it is defined as a quotient of the even tensor algebra by the two-sided ideal generated by elements of the form $m\otimes m - \psi_{\DR}(m,m)$.  Both the filtration and the Frobenius structure are induced from the filtration and Frobenius action on the even tensor algebra.  The fact that $\psi_{\DR} $ is compatible shows that they descend to the Clifford algebra.
%\Dcom{should say more here?  maybe turn into a lemma}

Similarly, for $\pi: \calAhat \rightarrow \Bhat$, we have the filtered Frobenius crystal $R^1\pi_*(\Omega^*_{\calAhat/\Bhat})$.  It inherits an action of the algebra $C^+$.  If we consider the coherent sheaf
$$\End_{C^{+}}(R^1\pi_*(\Omega^*_{\calAhat/\Bhat})),$$
it inherits a connection, filtration, and Frobenius action from the corresponding 
structure on the Frobenius crystal $\End(R^1\pi_*(\Omega^*_{\calAhat/\Bhat}))$.  

\begin{lem}\label{morita}
The above structures on the coherent sheaf
$$\End_{C^{+}}(R^1\pi_*(\Omega^*_{\calAhat/\Bhat}))$$
define a filtered Frobenius crystal.  Furthermore, if we reduce modulo $p$,
we have an isomorphism
$$\End_{C^{+}}(R^1\pi_*(\Omega^*_{\calAhat/\Bhat}))\otimes_{W} \kk \simrightarrow \End_{C^{+}}(R^1\pi_*(\Omega^*_{\calA_{\kk}/B_{\kk}}))$$
of filtered vector bundles with connection on $B_{\kk}$.
\end{lem}
\begin{proof}
In the second statement, the filtration on the right-hand side is induced from the filtration on
$$\End(R^1\pi_*(\Omega^*_{\calA_{\kk}/B_{\kk}}))$$
coming from the Hodge filtration on relative de Rham cohomology.

For the first statement, the only thing to check is that the coherent sheaf is locally free and its filtered pieces are locally direct summands.
For the second statement, we need to check compatibility of base change with both passing to the $C^{+}$-centralizer and the construction of the filtration.

We can argue as follows.  All statements are local, so after passing to an affine neighborhood, assume that
relative de Rham cohomology of $\pi$ and both its filtered piece and quotient are free on $\Bhat$, i.e. that we have a short exact sequence
$$0 \rightarrow F \rightarrow M \rightarrow Q \rightarrow 0$$
of free $\Ahat$-modules, all equipped with $C^{+}$-action.  By linearity, this action extends to an action of $C^{+}\otimes \Ahat$.

Since $p \nodivide 2d$, the symmetric pairing on $L_{2d}\otimes_{\ZZ} W$ is nondegenerate.  It follows from Knus (\cite{knus}, IV.3) 
that
$$C^{+}\otimes_{\ZZ} W = \Cl_{+}(L_{2d}\otimes_{\ZZ} W) \cong \mathrm{Mat}_{2^{19}}(W)$$
is a matrix algebra, and thus the same for $C^{+}\otimes_{\ZZ} \Ahat$.

We therefore have a Morita equivalence
$$\mu: (C^{+}\otimes_{\ZZ} \Ahat)\mathrm{-mod} \rightarrow  \Ahat\mathrm{-mod}$$
where the inverse functor is defined by tensoring with $\Ahat^{\oplus 2^{19}}$ and using the natural action of the matrix algebra.

In particular, it is easy to see from this that a $(C^{+}\otimes_{\ZZ} \Ahat)$-module $N$ is locally free as an $\Ahat$-module if and only if 
$\mu(N)$ is a locally free $\Ahat$-module.  

Using the equivalence $\mu$, we have an identification of $\Ahat$-modules
$$\End_{C^{+}}(M) = \End_{C^{+}\otimes\Ahat}(M) \simrightarrow \End_{\Ahat}(\mu(M))$$
which is locally free since $\mu(M)$ is.
Furthermore, the filtration on the left-hand-side is induced by the two-step filtration on $M$ by the prescription
$$\Fil^k(\End_{C^{+}}(M)) = \{\gamma \in \End_{C^{+}}(M) | \gamma(\Fil^j(M)) \subset \Fil^{j+k}(M)\}.$$
This corresponds via the above isomorphism with the analogous filtration on the right-hand-side (defined by $\mu(F)$).
Since $\mu(F)$ is locally free and is locally a direct summand, the filtered pieces of $\End_{\Ahat}(\mu(M))$ 
are also locally free and locally direct summands (one can see this, for example, by picking a local complement for $\mu(F)$).  This proves the first claim.

For the second claim, we use the fact that, for a $(C^{+}\otimes_{\ZZ} \Ahat)$-module $N$, we have
$$\mu(N\otimes_{W} \kk) = \mu(N)\otimes_{W} \kk$$
which can be checked using the formula for $\mu^{-1}$.  The isomorphism of vector bundles follows from the chain of isomorphisms
$$\End_{C^+}(M\otimes_{W} \kk) = \End_{\Ahat}(\mu(M\otimes_{W}\kk)) = \End_{\Ahat}(\mu(M)\otimes_{W}\kk) = \End_{\Ahat}(\mu(M))\otimes_{W}\kk = \End_{C^+}(M)\otimes_{W} \kk,$$
where use locally-freeness of $\mu(M)$ in the third step.
A similar argument works for the filtered pieces, e.g.
\begin{align*}
\Hom_{C^{+}}(Q\otimes_{W}\kk, F\otimes_{W}\kk) &= \Hom_{\Ahat}(\mu(Q)\otimes_{W}\kk, \mu(F)\otimes_{W}\kk) \\&= \Hom_{\Ahat}(\mu(Q),\mu(F))\otimes_{W}\kk = \Hom_{C^{+}}(Q, F)\otimes_{W}\kk.
\end{align*}
\end{proof}

%Since $\End_{C^{+}}(R^1\pi_*(\Omega^*_{\calAhat/\Chat}))$ is torsion free, the locus where it is not locally free is codimension $2$.  Since $C$ is two-dimensional, by removing finitely many points on $C_\kk$ before completion, we can arrange for this to hold from the beginning.
%Similarly, for the coherent sheaves $\Fil^{i}$, and $\mathrm{gr}^i(\Fil)$ these are subsheaves of the corresponding objects for $\End(R^1\pi_*(\Omega^*_{\calAhat/\Chat}))$.  Since the filtration has finitely many steps, we can argue by further shrinkage.

The key proposition that lets us prove quasifiniteness is the the following de Rham version of  equation \eqref{etaleclifford}.
\begin{prop}\label{clifford-crystal}
We have an isomorphism of filtered Frobenius crystals on $\Bhat$:
$$\Cl_{+}(P^2_{\DR}(f)\{1\}) \simrightarrow \End_{C^{+}}(R^{1}\pi_*(\Omega^*_{\calAhat/\Bhat})).$$
\end{prop}
We will prove this proposition using integral comparison theorems in $p$-adic Hodge theory, as we now explain.

\subsection{Relative $p$-adic Hodge theory}

It is easier to first explain the proof of this proposition after inverting $p$, using the rational version of the relative comparison theorem of \cite{faltings} (proven also in \cite{andreatta-iovita}).  
We will only work in the affine case, since everything can be made somewhat explicit in this setting.  Although the general definitions of crystalline sheaves and the comparison functor $\DD_{\Shat}$ are complicated, they can at least  be stated concretely in this case.

Let $R$ and $\Shat$ be as in Remark \ref{setup}.
Choose an algebraic closure $\Omega$ of the field of fractions $\mathrm{Frac}(R)$ and let $\Rbar$ be the union of the normalizations of $R$ in subfields $L \subset \Omega$, where $L$ ranges over finite extensions of 
 $\mathrm{Frac}(R)$ such that the normalization of $R[1/p]$ in $L$ is \'etale.
We have a canonical isomorphism 
$$\pi_1:= \pi_1(\Spec R[1/p], \Omega)= \mathrm{Gal}(\mathrm{Frac}(\Rbar)/\mathrm{Frac}(R)).$$

Let $\calB_{\cris}(\Rbar)$ denote the relative version of Fontaine's period ring, as defined in, e.g., Section $2.6$ of \cite{andreatta-iovita}.  We refer the reader to our references for the definition of this ring since it is very complicated.
It is a filtered $R[1/p]$-algebra, equipped with an action of $\pi_1$ above, a decreasing, exhaustive, and separated filtration by $R[1/p]$-submodules, an integrable connection $\Del$, and a Frobenius structure, such that the filtration is stable with respect to $\pi_1$ and the connection and Frobenius structure commute with this action. 

Given an \'etale $\QQ_p$-sheaf $\LL$ on $\Shat[1/p]$, having picked a base point $\Omega$, we can view it as a finite-dimensional $\QQ_p$ vector space $V$ equipped with a continuous action of $\pi_1$.  The (covariant) comparison functor $\DD_{\Shat}$ is given by the $R[1/p]$-module
$$\DD_{\Shat}(\LL) = (V \otimes_{\QQ_{p}} \calB_{\cris}(\Rbar))^{\pi_1}$$
where the Galois group acts via the diagonal action.  It inherits a filtration, connection, and Frobenius structure from those on $\calB_{\cris}(\Rbar)$.
\begin{defn}
The sheaf $\LL$ is crystalline if 
$\DD_{\Shat}(\LL)$ is a filtered F-isocrystal and
the natural map
$$\DD_{\Shat}(\LL) \otimes_{R} \calB_{\cris}(\Rbar) \rightarrow V \otimes_{\QQ_p} \calB_{\cris}(\Rbar)$$
is an isomorphism.
\end{defn}

We can now state Theorem $3.12$ of \cite{andreatta-iovita} (see also Theorem $2.6$ of \cite{faltings}), which relates
crystalline \'etale sheaves with filtered isocrystals.
\begin{prop}
Given $\Shat$ as in Remark \ref{setup}, $\DD_{\Shat}$ is a fully faithful functor from crystalline lisse \'etale $\QQ_p$-sheaves on $\Shat[1/p]$ to filtered F-isocrystals on $\Shat$.  It is compatible with tensor products, duals, and Tate twists.
Given a map $\iota: \widehat{S'} \rightarrow \Shat$ of such schemes over $W$, there is a natural equivalence of functors
$$\DD_{\widehat{S'}}\circ \iota^*_{\et} = \iota^*_{\DR}\circ\DD_{\Shat}$$ 
where $\iota^*_{\et}, \iota^*_{\DR}$ denote respectively base change functors on \'etale $\QQ_p$-sheaves and filtered F-isocrystals.
\end{prop}
%\Dcom{Do I need to name this equivalence?}

%We define the relative period ring $\calB_\cris(\Rbar)$ as follows.  First consider
%$S = \lim_{\backarrow} \Rbar/p$ where the maps in the projective system are each given by absolute Frobenius maps.  By assumption $S$ is perfect; if we consider the Witt ring
%$W(S)$, let $\epsilon = (\epsilon_n)$ where $\epsilon_n$ are the reductions mod $p$ of $p$-power s of $p$, and
%let $\xi = [\epsilon] - p$
%where $[\cdot]$ denotes Teichmuller lift.
%Let $A_\cris(\Rbar)$ denote the $p$-adic completion of the divided power envelope $D_\xi(W(S))$ of the ideal generated by $\xi$.  If we set $t = \log \xi$, then we have
%$$\calB_\cris(\Rbar)= A_\cris(\Rbar)[1/t, 1/p].$$
%In the case where $R = W$, these constructions recover the usual Fontaine period rings $A_\cris$ and $B_\cris$.

The remarkable feature of this functor is that, for proper smooth families, it takes \'etale cohomology to de Rham cohomology: (Prop. $4.1$ and Theorem $4.2$ in \cite{andreatta-iovita}; Theorem $6.3$ in \cite{faltings}).  Again, in what follows, $\Shat$ and $\widehat{S'}$
are as in Remark \ref{setup}.

\begin{prop}\label{comparison-isocrystal}
Given a proper smooth morphism $g: \widehat{Y} \rightarrow \Shat$, the \'etale $\QQ_p$ sheaf
$R_{\et}^{k}g_*(\QQ_p)$ is crystalline, and we have a natural isomorphism of filtered F-isocrystals
$$\Phi_g: R^{k}g_*(\Omega^*_{\widehat{Y}/\Shat})[1/p] \rightarrow \DD_{\Shat}(R_{\et}^{k}g_*(\QQ_p))$$
that is compatible with Chern classes of line bundles, Poincar\'e duality, pullback along morphisms $\widehat{S'} \rightarrow \Shat$, 
and pullback along morphisms $\widehat{Y'} \rightarrow \widehat{Y}$ between proper smooth schemes over $\Shat$.
\end{prop}
%\Dcom{tate twists too?}
With all this technology in place, proving Proposition \ref{clifford-crystal} on the level of isocrystals is easy.

We first show that the comparison functor $\DD_{\Shat}$ is compatible with both the Clifford construction and passing to $C^+$-invariants. 

\begin{lem}\label{clifford-iso-comp}
Given a crystalline \'etale $\QQ_p$-sheaf $\LL$ on $\Shat$ with symmetric pairing
$$\psi_{\et}:  \LL \otimes_{\QQ_p} \LL \rightarrow \QQ_p,$$
let $\calE = \DD_{\Shat}(\LL)$ be its associated filtered F-isocrystal and let
$$\psi_{\cris} = \DD_{\Shat}(\psi_{\et}): \calE \otimes_{\Shat[1/p]} \calE \rightarrow \calO_{\Shat[1/p]}$$
be the associated pairing on $\calE$.
Then there is a natural isomorphism of filtered F-isocrystals
$$\Cl_{+}(\calE, \psi_{\cris}) \simrightarrow \DD_{\Shat}( \Cl_{+}(\LL,\psi_{\et})).$$
\end{lem}
\begin{proof}
Note that to define $\psi_{\cris}$, we use the compatibility of $\DD_{\Shat}$ with tensor products.  
It follows from the formula for $\DD_{\Shat}$ and the definition of $\psi_{\cris}$ that there is a natural map of filtered F-isocrystals 
$$\Cl_{+}(\calE, \psi_{\cris}) \rightarrow \DD_{\Shat}(\Cl_{+}(\LL,\psi_{\et})).$$
Since the underlying sheaves of locally free $\Shat[1/p]$-modules have the same rank, it suffices to show this map is surjective.  Moreover, this map is compatible 
with the natural map
$$\bigoplus_{k = 0}^{\mathrm{rk}\LL/2}(\calE^{\otimes 2k}) \rightarrow \bigoplus_{k =0}^{\mathrm{rk}\LL/2}\DD_{\Shat}(\LL^{\otimes 2k}).$$
Since $\DD_{\Shat}$ is compatible with tensor product, this map is an isomorphism on each graded piece and, in particular surjective.  Since the Clifford algebra is a quotient of the above direct sum, we are done.
\end{proof}

We consider the abelian scheme $\pi: \calAhat \rightarrow \Bhat$, equipped with the action 
$C^+ \rightarrow \End_{\Bhat}(\calAhat)$.  Both the \'etale $\QQ_p$-sheaf $\Ret^{1}\pi_*(\QQ_p)$
and the filtered F-isocrystal $R^1\pi_*(\Omega^*_{\calAhat/\Bhat})[1/p]$ inherit actions of $C^+$.

We argue that these actions are compatible with comparison theorems.

\begin{lem}\label{abelian-iso-comp}
Given our setup above,
\begin{enumerate}
\item 
The action of $C^+$ on $\Ret^{1}\pi_*(\QQ_p)$ induces an action on $\DD_{\Bhat}(\Ret^{1}\pi_*(\QQ_p))$ such that the comparison isomorphism
$$\Phi_\pi: R^1\pi_*(\Omega^*_{\calAhat/\Bhat})[1/p] \simrightarrow \DD_{\Bhat}(\Ret^{1}\pi_*(\QQ_p))$$
intertwines the action of $C^+$.
\item 
There is a natural isomorphism
$$\Phi_\pi: \End_{C^+}(R^1\pi_*\Omega^*_{\calAhat/\Bhat})[1/p] \simrightarrow
\DD_{\Bhat} (\End_{C^+}(\Ret^1\pi_*(\QQ_p))).$$
\end{enumerate}
\end{lem}
\begin{proof}
Given $\gamma: \calAhat \rightarrow \calAhat$ over $\Bhat$, $\Phi_\pi$ is compatible with pullback along $\gamma$, by Proposition \ref{comparison-isocrystal}.  This gives the first claim.

For the second, since $\DD_{\Bhat}$ is compatible with duals and tensor products, we have
an isomorphism
$$\DD_{\Bhat}(\End(\Ret^1\pi_*(\QQ_p)) \rightarrow \End(\DD_{\Bhat}((\Ret^1\pi_*(\QQ_p)).$$
%\Dcom{It is easier for me to argue on the level of modules, and then later I can see how to write this cleanly in terms of sheaves.}
Consider $\Ret^1\pi_*(\QQ_p)$ as a vector space $V$ with an action of $\pi_1:= \pi_1(\mathrm{Spec}(\Bhat[1/p]),\bbar)$.
Given any element $\gamma \in C^+$, it gives an endomorphism of $V$ that commutes with the action of $\pi_{1}$.

Arguing as in Lemma \ref{morita}, we have
\begin{align*}
(\End_{C^+}(V) \otimes \calB(\Abar))^{\pi_1} &= (\End_{\calB}(V \otimes \calB(\Abar)))^{\pi_1,C^+} \\&= \left((\End_{\calB}(V\otimes \calB(\Abar)))^{\pi_1}\right)^{C^+} = 
\End_{C^+}((V\otimes\calB(\Abar))^{\pi_1}).
\end{align*}
\end{proof}

First, observe that $\Phi_f$ induces a natural isomorphism of filtered F-isocrystals
$$\Phi_f\{1\}:  P^2_{\DR}(f)\{1\}[1/p] \rightarrow \DD_{\Bhat}(\Pet\QQ_p(1)).$$
Indeed, this follows from the compatibility of $\Phi_f$ with taking Chern classes of $L$ and with taking the orthogonal complement with respect to the Poincar\'e pairing, as stated in Proposition \ref{comparison-isocrystal}.
Furthermore, $\Phi_f\{1\}$ intertwines the nondegenerate pairings
$\DD_{\Bhat}(\psi_{\et})$ and $\psi_{\DR}$.

We therefore have an isomorphism
\begin{align*}
\Cl_+(P^2_{\DR}(f)\{1\})[1/p] &= \Cl_+( P^2_{\DR}(f)\{1\}[1/p] ) = 
\Cl_+(\DD_{\Bhat}(\Pet\QQ_p(1))) \\&= \DD_{\Bhat}(\Cl_+(\Pet\QQ_p(1)))
= \DD_{\Bhat}(\End_{C^{+}}(R^1\pi_*\QQ_p))
=\End_{C^{+}}(R^1\pi_*\Omega^*_{\calAhat/\Bhat})[1/p].
\end{align*}

We have proven the following weak version of Proposition \ref{clifford-crystal}

\begin{lem}\label{clifford-isocrystal}
There exists an isomorphism of filtered F-isocrystals
$$\Psi_{\Bhat[1/p]}: \Cl_{+}( P^2_{\DR}(f)\{1\})[1/p] \simrightarrow \End_{C^{+}}(R^1\pi_*\Omega^*_{\calAhat/\Bhat})[1/p].
$$
\end{lem}

\subsection{Integral statements}\label{section-integral}
To complete the proof of Proposition \ref{clifford-crystal}, we need to show that the isomorphism in Lemma \ref{clifford-isocrystal} comes from an isomorphism of crystals.  This will come from the fact that the isomorphism \eqref{etaleclifford} is itself integral, and integral versions of the comparison theorems.

In fact, we only need these statements over $W$ rather than a general base, thanks to Lemma \ref{clifford-isocrystal} and the following lemma.  We are grateful to Brian Conrad for showing us a simpler argument for the following.  In what follows, let $R$ and $\Shat$ be as in Remark \ref{setup}.
\begin{lem}\label{brianlemma}
Given filtered Frobenius crystals $\calE$ and $\calF$ over $\Shat$, with an isomorphism 
$$\psi: \calE[1/p] \rightarrow \calF[1/p]$$
as filtered F-isocrystals on $\Shat[1/p]$.
Suppose that for every $W$-point
$$\iota: \mathrm{Spec}W \rightarrow \Shat,$$
the pullback isomorphism
$\iota^*\psi$
extends to an isomorphism
$$\Psi_{\iota}: \iota^*\calE \rightarrow \iota^*\calF$$
of filtered Frobenius crystals on $\Spec W$.
Then
$\psi$ extends to an isomorphism
$\Psi: \calE \rightarrow \calF$ of filtered Frobenius crystals on $\Shat$.
\end{lem}
\begin{proof}
It suffices to show that the isomorphism of locally free $R[1/p]$-modules extends uniquely to an isomorphism of $R$-modules, since the compatibility with $\Del$ and $\phi$ can be checked after inverting $p$.  We will show it extends to a map of $R$-modules; by applying the same result to the inverse (defined over $R[1/p]$), we see that the map will be an isomorphism.  Furthermore, by shrinking $R$, we can assume that $\calE$ and $\calF$ are free modules.

By studying matrix elements of the map, we are left to the following.  Given an element $f \in R[1/p]$ such that its specialization in $W[1/p]$ lies in $W$ for all maps of $W$-algebras $R \rightarrow W$, we want to show that $f \in R$.

Pick $r \geq 0$ such that $F = p^r f \in R$ and suppose $r > 0$.  Since $R$ is smooth, every $\kk$-point of $\Shat$ lifts to a $W$-point, for which the specialization of $F$ is divisible by $p$. In particular, $F \otimes \kk$ is contained in every maximal ideal of $R\otimes \kk$ which is a reduced algebra of finite type, so $F \in p R$, and we can replace $r$ with $r-1$.  Continuing inductively, we have that $f \in R$.
\end{proof}

\begin{defn} Given a crystalline representation $V$ of $\mathrm{Gal}(\overline{K}/K)$, its Hodge-Tate weights are the degrees of the nonzero graded pieces of the filtered F-isocrystal $\DD_{W}(V)$ with respect to its filtration.
\end{defn}

\begin{defn}
A filtered Frobenius crystal $M$ over $W$ with weights contained in $[0,b]$ is called a strongly divisible lattice if 
\begin{enumerate}
\item the associated isocrystal $M\otimes K$ is in the essential image of $\DD_{W}$,
\item $\phi(\Fil^k M) \subset p^{k}M$,
\item and $\sum_{k \geq 0} p^{-k}\phi(\Fil^k M) = M.$
\end{enumerate}
A filtered Frobenius crystal $M$ with weights contained in $[a,b]$ is an $a$-strongly divisible lattice if $M\{a\}$ is strongly divisible.
\end{defn}

Let us recall now a minor modification of some results in Fontaine-Laffaille theory, which are usually stated for weights in the range $[0,a]$, see \cite{BreuilMessing} for an overview.  The statement here is easily obtained from those by Tate twisting, using tensor product compatibility.  
A formula for the inverse functor is given in the Appendix, with more details provided in \cite{bhatt-snowden}.

\begin{prop}
Fix an interval of Hodge-Tate weights $[a,b]$ with $p > b-a+1$.
There exists an equivalence $\DDD_{[a,b]}$ between the category of Galois-stable lattices inside crystalline representations with Hodge-Tate weights contained in $[a,b]$ and the category of $a$-strongly divisible modules with the same range of weights.   This functor is compatible with tensor product of two lattices, provided that the tensor product also has Hodge-Tate weights in the range $[a,b]$.  We have the similar statement for duals.
\end{prop}

%A formula for this equivalence is given by $$\DDD_{[a,b]}(M) = (M(a)\otimes_{\ZZ_p} A_\cris)^{\mathrm{Gal}(K)} \otimes W\{-a\},$$where $A_\cris = A_\cris(\widehat{W})$ is 

As in the last section, we also have the integral comparison theorem of Fontaine-Messing \cite{FontaineMessing}, which relates \'etale and de Rham cohomology.
\begin{prop}
Given a smooth proper scheme $g: Y \rightarrow \Spec W$,
there is a natural isomorphism of filtered Frobenius crystals over $W$
$$\Phi^{\mathrm{int}}_g: H^k_{\DR}(Y/W) \rightarrow \DDD_{[0,k]}(H^k_{\et}(Y[1/p],\ZZ_p))$$
for $p > k+1$.
It is compatible with $\Phi_g$ in the sense that
$$\Phi^{\mathrm{int}}_g \otimes K = \Phi_g.$$
\end{prop}
The compatibility between Fontaine-Messing and Faltings's comparison isomorphisms is proven in \cite{niziol}.

We will only be interested in the range of weights $[-1,1]$ and will suppress that notation from now on: $ \DDD = \DDD_{[-1,1]}$.
%%%%%%%%%%%%%%%%%%%%%%%%%%%%%%%%% 

We need integral versions of the compatibility statements from Lemmas \ref{abelian-iso-comp} and \ref{clifford-iso-comp} from last section.  
In what follows, suppose we have an abelian scheme $\pi: A_W \rightarrow \Spec W$
equipped with a fiber-wise action of the algebra $C^+$.

\begin{lem}\label{abelian-comp}
Assume $p \geq 5$.
We have a natural isomorphism 
$$\Phi^{\mathrm{int}}_\pi: \End_{C^+}(R^1\pi_*(\Omega^*_{A_{W}/W})) \rightarrow
\DDD(\End_{C^+}(\Ret^1\pi_*(\ZZ_p))),$$
compatible with the rational isomorphism proven in the last section.
\end{lem}
\begin{proof}
It follows from the bound on the prime that
we have an isomorphism
$$\End(R^1\pi_*(\Omega^*_{A_{W}/W})) \simrightarrow
\DDD(\End(\Ret^1\pi_*(\ZZ_p))),$$
compatible with the rational isomorphism.
In combination with the rational statement in Lemma \ref{abelian-iso-comp}, we have the result.
\end{proof}

\begin{lem}\label{clifford-comp}
Assume $p \geq 5$.
Given a Galois-stable lattice $M$ of a crystalline representation $V$ over $\Spec W$ with symmetric pairing
$$\psi_{\et}:  M \otimes_{\ZZ_p} M \rightarrow \ZZ_p,$$
compatible with the Galois action, such that 
the Hodge-Tate weights of $V$ and $\Cl_+(V,\psi_{\et})$ are contained in $[-1,1]$.
Let $\calE = \DDD(M)$ be its associated strongly divisible module and let
$$\psi_{\cris} =\DDD(\psi_{\et}): \calE \otimes_{W} \calE \rightarrow W $$
be the associated pairing on $\calE$.
Then there is a natural isomorphism of strongly divisible modules
$$\Cl_{+}(\calE, \psi_{\cris}) \simrightarrow \DDD(\Cl_{+}(\LL,\psi_{\et}))$$
compatible with the rational isomorphism in the previous section.
\end{lem}
The rational proof of this lemma applies here integrally as written, but requires a much worse bound on $p$ in order for tensor product compatibility to apply.
To realize the stronger bound on the prime, we require some extra technology, and defer the proof to the Appendix.

We can now complete the proof of Proposition \ref{clifford-crystal}.
\begin{proof}
From Lemma \ref{clifford-isocrystal}, we have an isomorphism of filtered F-isocrystals
$$\Psi_{\Bhat[1/p]}: \Cl_{+}( P^2_{\DR}(f)\{1\})[1/p] \simrightarrow \End_{C^{+}}(R^1\pi_*\Omega^*_{\calAhat/\Bhat})[1/p] .$$
From Lemma \ref{brianlemma}, it suffices to show that
for every $\iota: \Spec W \rightarrow \Bhat$
$$\iota^*\Psi_{\Bhat[1/p]} = \Psi_\iota \otimes K$$
for some isomorphism of filtered F-crystals.
Let
$$f_{\iota}: X_W \rightarrow \Spec W, \qquad \pi_{\iota}: A_W \rightarrow \Spec W$$
denote the families over $W$ obtained via $\iota$.
We can construct $\Psi_\iota$ as before using
the sequence of isomorphisms
\begin{align*}
 \Cl_{+}( P^2_{\DR}(f_{\iota})\{1\}) \simrightarrow \Cl_{+}(\DDD(P^2_{\et}f_{\iota,*}(\ZZ_p(1)))) &\simrightarrow
 \DDD(\Cl_{+}(P^2_{\et}f_{\iota,*}(\ZZ_p(1))))\\ &\simrightarrow \DDD(\End_{C^+}(\Ret^1\pi_{\iota,*}(\ZZ_p)))
\simrightarrow \End_{C^+}(R^1\pi_*(\Omega^*_{A_{W}/W}))
\end{align*}
Compatibility with $\Psi_{\Bhat[1/p]}$ then follows from the other compatibilities listed in the above lemmas.
\end{proof}

\subsection{Proof of quasifiniteness}

We can now finish the proof of Proposition \ref{quasifinite} and thus positivity of the Hodge bundle.

Using Proposition \ref{clifford-crystal}, if we restrict filtered Frobenius crystals to $B_{\kk}$, we have an isomorphism of
filtered vector bundles with integrable connection
\begin{equation}\label{derhammodp}
\Cl_{+}(P^2_{\DR}(f_{\kk})\{1\}) = \Cl_{+}(P^2_{\DR}(f)\{1\}) \otimes \kk  \simrightarrow \End_{C^{+}}(R^1\pi_{*}\Omega^*_{\calAhat/\Bhat}) \otimes \kk = \End_{C^{+}}(R^1\pi_{*}\Omega^*_{\calA_{\kk}/B_{\kk}}).
\end{equation}

Here we have used Lemma \ref{morita} on the right-hand side.

We now state the following simple lemma:

%Since the filtered pieces and successive quotients of $\End_{C^{+}}(R^1\pi_{*}\Omega^*_{\calAhat/\Chat})$ are locally free, 
%we have an inclusion 
%$$0 \rightarrow \End_{C^{+}}(R^1\pi_{*}\Omega^*_{\calAhat/\Chat}) \otimes \kk \rightarrow \End(R^1\pi_{*}\Omega^*_{\calAhat/\Chat}) \otimes \kk = \End(R^1\pi_{\kk,*}(\Omega^*_{\calAhat_{\kk}/C_{\kk}})$$
%of filtered vector bundles with integrable connection, meaning that the filtration of the left-hand side is the subspace filtration.
%remove all graded; work with Del directly
\begin{lem}
Let $E$ and $E'$ be vector bundles on $B_{\kk}$, equipped with decreasing filtrations $\Fil$ and $\Fil'$, integrable connections $\Del$ and $\Del'$, and an inclusion 
$$0 \rightarrow E' \rightarrow E$$
compatible with connections such that $\Fil^{k,'} = \Fil^k \cap E'$.
Suppose that the restriction of $\Del$ preserves a given filtered piece
$$\Del: \Fil^{k} \rightarrow \Fil^{k} \otimes \Omega^1_{B_{\kk}}$$
then the same holds for $\Del'$ and $\Fil^{k,'}$.
\end{lem}

Since $\pi_{\kk}:\calA_{\kk} \rightarrow B_{\kk}$ is trivial, the Gauss--Manin connection on
$$\End(R^1\pi_{*}\Omega^*_{\calA_{\kk}/B_{\kk}})$$
is trivial and in particular preserves every step of the Hodge filtration.
Therefore, using the above lemma and equation \eqref{derhammodp}, it follows that
the connection $\Del$ on 
$$\Cl_{+}(P^2_{\DR}(f_{\kk})\{1\})=\End_{C^{+}}(R^1\pi_{*}\Omega^*_{\calA_{\kk}/B_{\kk}})$$
preserves each step of the filtration.  In particular, $\Del$ preserves the filtered piece $\Fil^1(\Cl_{+}(P^2_{\DR}(f_{\kk})\{1\}))$.

On the other hand, suppose the Kodaira-Spencer map
$\mathsf{gr}^{2}\Del$ for $f_{\kk}$ 
is nonzero at some closed point $b \in B_{\kk}$.
We can pass to primitive cohomology and rephrase this as saying that the connection 
$\Del$ does not preserve the filtered piece $\Fil^1(P^2_{\DR}(f_{\kk})\{1\})$.

Replace $B_{\kk}$ by an affine neighborhood such that $P^2_{\DR}(f_{\kk})\{1\})$ and all filtered pieces
are given by free modules and such that there exists an everywhere-nonzero vector field $v$.

Then if we choose a basis vector $\omega$ for the rank-one module $\Fil^1(P^2_{\DR}(f_{\kk})\{1\})$,
we have that
$$\eta := \Del_{v}\omega \notin \Fil^1(P^2_{\DR}(f_{\kk})\{1\}).$$
By Griffiths transversality, we have that $\eta \in \Fil^0(P^2_{\DR}(f_{\kk})\{1\})$; after possibly passing to a smaller neighborhood, we can 
complete $\{\omega, \eta\}$ to a filtered basis of $P^2_{\DR}(f_{\kk})\{1\})$:
$$\omega, \eta_1 = \eta, \eta_2, \dots, \eta_{19}, \gamma$$
with $\eta_k \in \Fil^0(P^2_{\DR}(f_{\kk})\{1\})$.

Since $\omega, \eta,$ and $\eta_2$ are linearly independent, we know
\begin{equation}\label{notdivisible}
\eta\cdot \eta_2 \notin \omega\cdot \Cl(P^2_{\DR}(f_{\kk})\{1\}).
\end{equation}
If we apply $\Del_v$ to $\omega\cdot \eta_2$, we see that 
$$\eta\cdot\eta_2 =  \Del_v(\omega\cdot\eta_2) - \omega \cdot \Del_{v}\eta_2.$$
Using the description of the Clifford filtration in Example \ref{clifford-dR}, we have that
$$\omega\cdot\eta_2 \in \Fil^1(\Cl_{+}(P^2_{\DR}(f_{\kk})\{1\}))$$
and therefore also
$$\Del_v(\omega\cdot\eta_2) \in \Fil^1(\Cl_{+}(P^2_{\DR}(f_{\kk})\{1\})) \subset  \omega\cdot \Cl(P^2_{\DR}(f_{\kk})\{1\}).$$
This implies that $\eta\cdot\eta_2$ is divisible by $\omega$, so we have a contradiction with equation \eqref{notdivisible}.

\section{Proof of Main theorem}

Let $k$ be an algebraically closed field of characteristic $p$, and $X/k$ a supersingular K3 surface 
with a polarization $L$ of degree $2d$ with $p> 2d+4$.  We can now prove Artin's conjecture for $X$, along the lines sketched in the introduction.

We fix $n=4$ and work with the moduli space $\M_{2d,n}$ with spin level structure.  As always, let $W = W(k)$ denote the ring of Witt vectors with fraction field
$K= W[1/p]$.

We first construct a proper one-dimensional family of supersingular K3 surfaces containing $X$.
By Theorem $15$ of \cite{ogus}, the supersingular locus of $\M^{\circ}_{2d,n,k}$ is a closed algebraic subspace of dimension $9$.  In particular,
there exists a nontrivial map
$$\iota: C^{\circ} \rightarrow \M^{\circ}_{2d,n,k}$$
from an affine open subset $C^{\circ}$ of a smooth proper curve $C/k$, whose image is contained in the supersingular locus.
Let
$$f: \calX^{\circ} \rightarrow C^{\circ}$$
be the associated polarized family of K3 surfaces which we can assume, after taking an \'etale cover, carries a relatively ample line bundle $L$.

By Theorem $5.2$ of \cite{saint-donat}, given a K3 surface over an algebraically closed field of odd characteristic, with an ample line bundle $L$ that is not 
very ample, either the surface has a polarization of degree $2$ or the surface is elliptic.
In these cases, Artin's conjecture holds by either \cite{rsz} or \cite{artin-ss}.
Therefore, we can assume that $L$ is very ample on the generic fiber of $f$ and we can apply Theorem \ref{ss-degeneration}.  
The following lemma uses the local result to compactify the family $f$.

\begin{lem}
After possibly replacing $C^\circ$ by a finite cover, we can compactify $f$ to a quasipolarized family of supersingular K3 surfaces
with spin level $n$ structure
$$f: \calX \rightarrow C$$
over a smooth, proper, connected curve $C$.
\end{lem}
\begin{proof}

Given each closed point $c$ in the (finite) complement $C\backslash C^\circ$, let 
$$\Delta_c = \Spec \calO_{C,c}.$$
By Theorem \ref{ss-degeneration}, there exists a finite, separable base change
$$\Delta'_{c} \rightarrow \Delta_c$$
and a map $\Delta'_{c} \rightarrow \M_{2d,k}$ 
extending the restriction of $\iota$ to the generic point of $\Delta_c$.

For each $c$, the generic point of $\Delta'_{c}$ defines a finite separable field extension of the function field $k(C)$.
Choose a finite cover $C' \rightarrow C$
such that the field extension
$k(C) \rightarrow k(C')$
of function fields is a finite separable extension containing all these finite extensions of $k(C)$.
Given a point $t \in C'$ lying over a boundary point $b$, let $\Delta_t = \Spec \calO_{C',t}$ and let $\Delta^{\circ}_t$ be the
complement of the closed point.  By construction, 
 $$\Delta^{\circ}_t = \Spec k(C') \subset \Delta_t \rightarrow \Delta_{c}$$
 lifts to a map
 $$\Delta^{\circ}_t = \Spec k(C') \rightarrow \Delta'_{c}.$$
By properness, this extends to
$$\Delta_t \rightarrow \Delta'_c \rightarrow \M_{2d,k}.$$
Therefore, we can compactify the pullback of $\iota$ on the preimage of $C^\circ$ to a map $C' \rightarrow \M_{2d,k}.$
After taking a finite \'etale cover of $C'$, we can construct the spin level $n$ structure as well.
\end{proof}

We now apply the results of Section \ref{sectionpicardjumping}.  If we take Theorem \ref{Picardjumping} applied to the rank $2$ lattices associated to elliptic surfaces (as in the proof of Corollary \ref{elliptic-case}), and pullback the linear equivalence constructed there
to $\M_{2d,n,\CC}$, we have
a divisor $D_{\CC}$ consisting of elliptic K3 surfaces
such that
\begin{equation}\label{hodgeelliptic}
\lambda^{\otimes a} = \calO(D_{\CC}) \in \Pic(\M_{2d,n,\CC})
\end{equation}
where $\lambda$ is the Hodge bundle on $\M_{2d,n}$ and $a > 0$.

Given a rank $2$ lattice $\Lambda$ of the type considered in Section \ref{sectionpicardjumping}, the divisors $D_{\Lambda}$ are defined over $\overline{\QQ}$, since they can be described as the images of moduli spaces of quasipolarized K3 surfaces equipped with an extra line bundle.  Furthermore, the Galois conjugate of an elliptic K3 surface over $\overline{\QQ}$ is still elliptic.
Therefore, arguing as the first paragraph of Proposition \ref{hodgecompare}, equation \eqref{hodgeelliptic} descends to $\QQ$ perhaps after passing to a multiple and replacing $D$ with a union of conjugates.  That is, there exists a Cartier divisor $$D_{\QQ} \subset \M_{2d,n,\QQ},$$ whose geometric points correspond to elliptic K3 surfaces, such that 
$$\lambda^{\otimes a} = \calO(D_{\QQ}) \in \Pic(\M_{2d,n,\QQ}).$$

We can base change to $K$, and let $D\subset \M_{2d,n,W}$ be the Cartier divisor obtained by taking the closure of $D\otimes K \subset \M_{2d,n,K}$, with multiplicities.
Applying Lemma \ref{specialize-line}, the equality of line bundles on $\M_{2d,n,K}$ specializes to an isomorphism over $\M_{2d,n,k}$, i.e. we have an isomorphism of line bundles on $\M_{2d,n,k}$
$$\lambda^{\otimes a} = \calO(D\otimes k) \in \Pic(\M_{2d,n,k}).$$

By Theorem \ref{amplehodge}, 
$$\deg_{C} \lambda^{\otimes a} > 0,$$
so we must have
$$C \cap (D\otimes k) \ne \emptyset.$$
In particular there exists at least one closed fiber $\calX_t$ of $f$ which is a supersingular elliptic K3 surface.
By Theorem $1.7$ of \cite{artin-ss}, $\calX_t$ has Picard rank $22$.  Since the Picard rank of a supersingular K3 surface is constant in connected families, by Theorem $1.1$ of \cite{artin-ss},
$X$ has Picard rank $22$ as well.

Finally, notice that the only place where $p > 2d+4$ is used is Section \ref{saito-section} to construct a semistable model via Saito's work.  If we assume semistable reduction for
surfaces over a one-dimensional base, then the rest of the paper only requires $p \nodivide 2d$.  Therefore, under this assumption, we have an improved bound on the prime.

\appendix

\section{Compatibility of Clifford constructions with Fontaine--Laffaille functor (by A. Snowden)}

Let $p>2$ be a prime, let $K/\QQ_p$ be an unramified extension with absolute Galois group $G$ and ring of integers $W$.
Let $D$ be a filtered $\varphi$-module with Hodge-Tate weights in $[a, b]$ with $b-a<p$ (these are called filtered F-isocrystals in Section \ref{section-setup}).  
By an \emph{$a$-strongly divisible
lattice} in $D$, we mean a $W$-lattice $M$ in $D$ such that $M\{a\}$ is a strongly divisible lattice in $D\{a\}$ in
the usual sense.  We let $\FL_{[a,b]}$ be the Fontaine--Laffaille functor, which takes $a$-strongly divisible lattices
in $D$ bijectively to $G$-stable lattices in the corresponding Galois representation.  It is defined by
\begin{displaymath}
\FL_{[a,b]}(M)=\Hom_{\Fil,\varphi}(M\{a\}, A_{\cris}(a))^{\vee}.
\end{displaymath}
With this convention, we have $\FL_{[a, b]}(W\{k\})=\ZZ_p(k)$ for any $k \in [a, b]$; the inverse functor is $\DDD_{[a,b]}$ in Section \ref{section-integral}.   We follow here the Tate twist conventions from that section.

The goal of this appendix is
to prove the following result.

\begin{prop}
\label{app1}
Let $a<b$ be integers with $b-a<p-1$.
\begin{itemize}
\item Let $V$ be a crystalline representation with weights in $[a, b]$.
\item Let $\psi_V:V \otimes V \to \QQ_p$ be a symmetric Galois compatible pairing.
\item Let $D$ be the filtered $\varphi$-module corresponding to $V$.
\item Let $\psi_D:D \otimes D \to K$ be the pairing corresponding to $\psi_V$.
\item Let $T \subset V$ be a $G$-stable lattice such that the restriction $\psi_T$ of $\psi_V$ to $T$ is $\ZZ_p$-valued.
\item Let $M$ be the $a$-strongly divisible lattice in $D$ corresponding to $T$ under $\FL_{[a, b]}$.
\end{itemize}
Assume that the Clifford algebra $\Cl(V, \psi_V)$ has weights in $[a, b]$.  Then the restriction $\psi_M$ of $\psi_D$
to $M$ is $W$-valued and the Clifford algebra $\Cl(M, \psi_M)$ is the $a$-strongly divisible lattice in
$\Cl(D, \psi_D)$ corresponding to $\Cl(T, \psi_T)$ under $\FL_{[a, b]}$.  Similarly, the even Clifford algebra
$\Cl_+(M, \psi_M)$ is the $a$-strongly divisible lattice in $\Cl_+(D, \psi_D)$ corresponding to $\Cl_+(T, \psi_T)$
under $\FL_{[a, b]}$.
\end{prop}

\begin{remark}
In the application of this proposition, $V$ is of dimension 21 and both $V$ and $\Cl(V, \psi_V)$ have weights in
$[-1, 1]$.  It is not difficult to show, using only Fontaine--Laffaille theory,
that formation of the Clifford algebra is compatible with Fontaine--Laffaille theory for $p>42$.  The proposition
gives compatibility for $p \ge 5$.
\end{remark}

To prove the result, we will need to make use of results of Kisin.  Let $\mf{S}=W [[u]]$, equipped with the
Frobenius map $\varphi$ extending the natural one on $W$ and taking $u$ to $u^p$.  Let $E(u) \in \mf{S}$ be the
Eisenstein polynomial $u-p$.  A \emph{Kisin module} is a free $\mf{S}$-module $\mf{M}$ of finite rank equipped with a
$\mf{S}$-linear isomorphism $\varphi:\varphi^*(\mf{M})[1/E(u)] \to \mf{M}[1/E(u)]$ (see \cite[\S 4.1]{Kisin2}).  The
notion of a morphism of Kisin modules is evident.  The kernel of a surjection of Kisin modules is again a Kisin module;
the corresponding result for cokernels is clearly not true.  The category of Kisin modules is stable under tensor
products and duality. Let $K_{\infty}$ be the extension of $K$ obtained by adjoining a compatible system
of $p$-power roots of $p$, and let $G_{\infty} \subset G$ be its absolute Galois group.  Given a Kisin module $\mf{M}$,
Kisin constructed a finite free $\ZZ_p$-module $\uT(\mf{M})$ equipped with an action of $G_{\infty}$.  In
the ``effective'' case, where $\varphi(\mf{M}) \subset \mf{M}$, this is defined by
\begin{displaymath}
\uT(\mf{M})=\Hom_{\mf{S}, \varphi}(\mf{M}, \mf{S}^{\rm ur})^{\vee},
\end{displaymath}
see \cite[\S 2.1.4]{Kisin} (though note we have added a dual); in general, $\uT(\mf{M})$ is defined by twisting to
the effective case, using the above definition and then untwisting, much like $\FL_{[a, b]}$ is defined.  The functor
$\uT$ is fully faithful, compatible with duality, preserves surjections and takes exact sequences to
exact sequences.  These properties follow from \cite[\S 2.1.4]{Kisin}, \cite[\S 2.1.12]{Kisin} together with
Fontaine's theory of \'etale $\mc{O}_{\mc{E}}$-modules (see \cite[\S 2.2]{LiuFontaine} for a summary).  Kisin
showed (\cite[\S 2.1.5]{Kisin} and \cite[\S 2.1.15]{Kisin}) that every $G$-stable lattice in a semi-stable
representation is of the form $\uT(\mf{M})$ for some Kisin module $\mf{M}$.

\begin{lem}
\label{app2}
Let $V$ be a crystalline representation of $G$ with weights in $[a, b]$ with $b-a<p-1$ and let $D$ be the filtered
$\varphi$-module corresponding to $V$.  Let $T$ be a $G$-stable lattice in $V$ and let $\mf{M}$ be the Kisin module
associated to $T$.  Then the $\varphi^*(\mf{M}/u\mf{M})$ is naturally the $a$-strongly divisible lattice in $D$
corresponding to $T$ under $\FL_{[a, b]}$.  (Here $\varphi^*$ indicates that the $W$-module structure is twisted by
$\varphi$.)
\end{lem}

\begin{proof}
By twisting, it suffices to treat the case
$a=0$ and $b=p-2$.  Let $S$ be Breuil's ring \cite[\S 2.2]{LiuBreuil}.  Let $\mc{D}=D \otimes_K
S[1/p]$ be the rational Breuil module assocaited to $D$.  The main result of \cite{LiuBreuil} shows that there exists
a unique strongly divisible lattice $\mc{M}$ in $\mc{D}$ such that $\uT(\mc{M})=T$, where
\begin{displaymath}
\uT(\mc{M})=\Hom_{S, \varphi, N, \Fil}(\mc{M}, \wh{A}_{\st})^{\vee}.
\end{displaymath}
Furthermore, Liu shows that $\mc{M}$ is canonically identified with $\mf{M} \otimes_{\mf{S}, \varphi} S$.  Now,
one easily checks that $M \otimes_W S$ is a strongly divisible lattice in $\mc{D}$, and
\begin{displaymath}
\uT(M \otimes_W S)=\Hom_{S, \varphi, N, \Fil}(M \otimes_W S, \wh{A}_{\st})^{\vee}=\Hom_{\varphi, \Fil}(M,
\wh{A}_{\st}^{N=0})^{\vee}.
\end{displaymath}
As $\wh{A}_{\st}^{N=0}=A_{\cris}$, we obtain $\uT(M \otimes_W S)=\FL_{[a, b]}(M)=T$.  By uniqueness, $M \otimes_W S=
\mc{M}=\mf{M} \otimes_{\mf{S}, \varphi} S$.  Applying $-\otimes_S W$ (where the map $S \to W$ sends $u$ to 0),
we find $M=\varphi^*(\mf{M}/u\mf{M})$.
\end{proof}

\begin{lem}
\label{app3}
Let $\mf{M}$ be a Kisin module and let $\psi_{\mf{M}}:\mf{M} \otimes \mf{M} \to \mf{S}$ be a symmetric map of Kisin
modules.  Let $T=\uT(\mf{M})$ and let $\psi_T$ be the induced pairing on $T$.  Then $\Cl(\mf{M}, \psi_{\mf{M}})$
is a Kisin module, and there is a natural isomorphism $\uT(\Cl(\mf{M}, \psi_{\mf{M}})) \to \Cl(T, \psi_T)$.
\end{lem}

\begin{proof}
For a $\ZZ$-module $M$, let $\mc{T}(M)$ denote the tensor algbera on $M$ and $\mc{T}_n(M)$ the truncated tensor algebra
$\bigoplus_{k \le n} M^{\otimes k}$, which we regard as a subgroup of $\mc{T}(M)$.  Let $\mf{C}_n$ (resp. $\mf{I}_n$)
denote the image (resp.\ kernel) of the map $\mc{T}_n(\mf{M}) \to \Cl(\mf{M}, \psi_{\mf{M}})$.
There is a natural isomorphism of $\mf{C}_n/\mf{C}_{n-1}$
with $\bw{n}{\mf{M}}$, which shows, inductively, that $\mf{C}_n$ is free as an $\mf{S}$-module.  It is clear that
the Frobenius on $\mf{M}$ induces one on $\mf{C}_n$, and so $\mf{C}_n$ is a Kisin module.  In particular,
$\Cl(\mf{M}, \psi_{\mf{M}})=\mf{C}_r$ is a Kisin module, where $r$ is the rank of $\mf{M}$

We have an exact sequence of Kisin modules
\begin{displaymath}
0 \to \mf{I}_n \to \mc{T}_n(\mf{M}) \to \mf{C}_n \to 0.
\end{displaymath}
Applying $\uT$ and taking the direct limit over $n$, we see that there is an exact sequence
\begin{displaymath}
0 \to \varinjlim \uT(\mf{I}_n) \to \mc{T}(T) \to \uT(\Cl(\mf{M}, \psi_{\mf{M}})) \to 0.
\end{displaymath}
It is easy to see that the right map is an algebra homomorphism, and so the group on the left is a 2-sided ideal of
$\mc{T}(T)$.  We claim that it is generated by $\uT(\mf{I}_2)$.  To see this, note that the ideal $\varinjlim \mf{I}_n$
of $\mc{T}(\mf{M})$ is generated by $\mf{I}_2$.  It follows that the map of Kisin modules
\begin{displaymath}
\bigoplus_{i+j=n-2} \left( \mc{T}_i(\mf{M}) \otimes \mf{I}_2 \otimes \mc{T}_j(\mf{M}) \right) \to \mf{I}_n
\end{displaymath}
is surjective.  Applying $\uT$, and using that $\uT$ is a tensor functor which
preserves surjections, establishes the claim.

Now, we have an isomorphism
\begin{displaymath}
\Sym^2(\mf{M}) \to \mf{I}_2, \qquad xy \mapsto x \otimes y + y \otimes x - 2 \psi_{\mf{M}}(x, y).
\end{displaymath}
Applying $\uT$, and using the compatibility of $\uT$ with $\Sym^2$ (which is obvious since $p \ne 2$), we see
that $\uT(\mf{I}_2)$ is exactly the kernel of the map $\mc{T}_2(T) \to \Cl(T, \psi_T)$.  We have thus shown that
$\ker(\mc{T}(T) \to \uT(\Cl(\mf{M}, \psi_{\mf{M}})))$ and $\ker(\mc{T}(T) \to \Cl(T, \psi_T))$ are two 2-sided
ideals of $\mc{T}(T)$ which are generated by their intersections with $\mc{T}_2(T)$, and that these intersections
coincide.  It follows that the ideals coincide, and so $\uT(\Cl(\mf{M}, \psi_{\mf{M}}))$ is naturally isomorphic to
$\Cl(T, \psi_T)$.
\end{proof}

We now prove the proposition:

\begin{proof}[Proof of Proposition~\ref{app1}]
Let $\mf{M}$ be the Kisin module such that $\uT(\mf{M})=M$.  Since $\uT$ is a fully faithful tensor functor, the
pairing $\psi_M$ comes from a symmetric pairing $\psi_{\mf{M}}$ on $\mf{M}$.  By Lemma~\ref{app3}, the Clifford
algbera $\mf{C}=\Cl(\mf{M}, \psi_{\mf{M}})$ is a Kisin module, and $\uT(\mf{C})$ is naturally
identified with $\Cl(T, \psi_T)$.  Since $\Cl(V, \psi_V)$ has weights in $[a, b]$, it follows from Lemma~\ref{app2}
that $\varphi^*(\mf{C}/u\mf{C})$ is naturally identified with the strongly divisible lattice corresponding to
$\Cl(T, \psi_T)$ under $\FL_{[a, b]}$.  Since formation of Clifford algebras is compatible with base change,
$\mf{C}/u\mf{C}$ is the Clifford algebra associated to $\mf{M}/u\mf{M}$ with respect to the pairing induced by
$\psi_{\mf{M}}$.  Applying Lemma~\ref{app2} again, we see that $\varphi^*(\mf{M}/u\mf{M})$ is naturally isomorphic to
$M$; under this identification, the pairing induced by $\psi_{\mf{M}}$ corresponds to $\psi_M$, which shows that
$\psi_M$ takes values in $W$.  We have thus shown that $\uT(\Cl(M, \psi_M))$ is identified with
$\Cl(T, \psi_T)$.  The statement for even Clifford algebras follows, as the even Clifford algebra is obtained
from the full Clifford algebra by taking $\ZZ/2\ZZ$-invariants.

This shows that $\Cl(M, \psi_M)$ is naturally identified with the
$a$-strongly divisible lattice in $\Cl(D, \psi_D)$ which corresponds to $\Cl(T, \psi_T)$ under $\FL_{[a, b]}$.
The last thing to check is that these two sublattices of $\Cl(D, \psi_D)$ are \emph{equal}.  
This is not difficult but requires one to verify that various identifications are compatible.
Since details of this type of verification will be found in \cite{bhatt-snowden}, we sketch the argument here.  
The basic idea is to work with Kisin modules inside a fixed rational Breuil module.  If $\mc{D}$ is an admissible Breuil module and
$T$ is a $G_{\infty}$-stable lattice in the corresponding Galois representation $V$, there is a unique Kisin module
$\mf{M}$ embedded in $\mc{D}$ such that $\uT(\mf{M})=M$ under the fixed identification of $\uT(\mf{M})[1/p]$ with $V$.
One can then verify, as in the above proof, that $\Cl(\mf{M}, \psi_{\mf{M}}) \subset \Cl(\mc{D}, \psi_{\mc{D}})$
corresponds to $\Cl(T, \psi_T)$, and this gives the desired equality.  We refer the reader to the upcoming paper \cite{bhatt-snowden} for a more detailed argument along these
lines.
\end{proof}

\hskip\baselineskip

\end{document}